\documentclass[a4paper, 11pt]{article}
\usepackage[T1]{fontenc}
\usepackage[utf8]{inputenc}
\usepackage{amsmath}
\usepackage{amsfonts}
\usepackage{amssymb}
\usepackage{amsthm}
\usepackage{graphicx}
\usepackage[english]{babel}
\usepackage{version}
\usepackage{nicefrac}
\usepackage{tipa} 
\usepackage{hyperref}

\theoremstyle{definition}
\newtheorem{defin}{Definition}[section]
\newtheorem{ex}[defin]{Example}
\theoremstyle{plain}
\newtheorem{theo}[defin]{Theorem}
\newtheorem{lemma}[defin]{Lemma}
\newtheorem{obs}[defin]{Remark}
\newtheorem{prop}[defin]{Proposition}
\newtheorem{cor}[defin]{Corollary}
\newtheorem{theorem}{Theorem}

\newtheorem*{theorem-no}{Theorem}

\numberwithin{equation}{section}

\renewenvironment{abstract}
{\par\noindent\textbf{\abstractname.}\ \ignorespaces}
{\par\medskip}

\title{Continuity of critical exponent of quasiconvex-cocompact groups under Gromov-Hausdorff convergence}
\author{Nicola Cavallucci}
\date{}

\begin{document}
	\maketitle
	\begin{abstract}
		\footnotesize
		We show continuity under equivariant Gromov-Hausdorff convergence of the critical exponent of discrete, non-elementary, torsion-free, quasiconvex-cocompact groups with uniformly bounded codiameter acting on uniformly Gromov-hyperbolic metric spaces.
	\end{abstract}

\tableofcontents

\section{Introduction}
\label{sec-introduction}
The \emph{critical exponent} of a discrete group of isometries $\Gamma$ of a proper metric space $X$ is defined as 
\begin{equation*}
	\label{intro-critical-exponent}
	h_\Gamma := \limsup_{T\to +\infty}\frac{1}{T}\log \Gamma x \cap \overline{B}(x,T),
\end{equation*}
where $x$ is any point of $X$. In \cite{Cav21bis} the author proved that if $X$ is a Gromov-hyperbolic space then the limit superior above is a true limit (see also Lemma \ref{entropy-critical}). \vspace{2mm}

\noindent A discrete group $\Gamma$ of isometries of a proper, $\delta$-hyperbolic metric space is said quasiconvex-cocompact if it acts cocompactly on the quasiconvex-hull of its limit set $\Lambda(\Gamma)$, namely QC-Hull$(\Lambda(\Gamma))$. In this case the codiameter is by definition the diameter of the quotient metric space $\Gamma\backslash$QC-Hull$(\Lambda(\Gamma))$.\\
In the sequel we denote by $\mathcal{M}(\delta, D)$ the class of triples $(X,x,\Gamma)$ where $X$ is a proper, $\delta$-hyperbolic metric space, $\Gamma$ is a discrete, torsion-free, non-elementary, quasiconvex-cocompact group of isometries of $X$ with codiameter $\leq D$ and $x$ is a point of QC-Hull$(\Lambda(\Gamma))$.
We refer to Section \ref{sec-Gromov} for the details of all these definitions. \\
We are interested in convergence of sequences of triples in $\mathcal{M}(\delta,D)$ in the equivariant pointed Gromov-Hausdorff sense, as defined by Fukaya in \cite{Fuk86}. It is a version of the classical pointed Gromov-Hausdorff convergence that consider also the groups acting on the spaces. Its precise definition is recalled in Section \ref{sec-convergence}. Our main result is:

\begin{theorem}
	\label{theorem-main}
	Let $\delta, D\geq 0$ and let $(X_n,x_n,\Gamma_n)_{n\in\mathbb{N}} \subseteq \mathcal{M}(\delta,D)$. If the sequence $(X_n,x_n,\Gamma_n)$ converges in the equivariant pointed Gromov-Hausdorff sense to $(X_\infty, x_\infty, \Gamma_\infty)$ then 
	\begin{itemize}
		\item[(i)] $(X_\infty, x_\infty, \Gamma_\infty) \in \mathcal{M}(\delta,D)$ and
		\item[(ii)] $h_{\Gamma_\infty} = \lim_{n\to +\infty}h_{\Gamma_n}$.
	\end{itemize}
\end{theorem}
The first difficulty in the proof of (i) is to show that the limit group $\Gamma_\infty$ is discrete. The proof is based on a result of \cite{BCGS2}: if $(X,x,\Gamma)\in \mathcal{M}(\delta,D)$ satisfies $h_\Gamma \leq H < +\infty$, then the global systole of $\Gamma$ is bigger than some positive constant depending only on $\delta,D$ and $H$ (cp. Proposition \ref{prop-bound-systole}).
This is a powerful tool when used together with Corollary \ref{cor-entropy-bound-convergence}: under the assumptions of Theorem \ref{theorem-main} the critical exponents of the groups $\Gamma_n$ are uniformly bounded above by some $H<+\infty$.
All the assumptions on the class $\mathcal{M}(\delta,D)$ are necessary in order to get the discreteness of the limit group, see Section \ref{subsec-counterexamples}. The second difficulty is to show that $\Gamma_\infty$ is quasiconvex-cocompact. In order to do that we will show that the limit of the Gromov boundaries $\partial X_n$ can be seen as a canonical subset of $\partial X_\infty$, Proposition \ref{boundary-convergence}. Under this identification the limit of the sets $\Lambda(\Gamma_n)$ coincides with $\Lambda(\Gamma_\infty)$.\vspace{2mm}

\noindent The proof of the continuity statement, Theorem \ref{theorem-main}.(ii), is based on the following uniform equidistribution of the orbits. It is a quantified version of a result of Coornaert, \cite{Coo93}.
\begin{theorem}
	\label{theo-uniform-distribution}
	Under the assumptions of Theorem \ref{theorem-main} there exists $K > 0$ such that for every $n$ and for every $T\geq 0$ it holds
	$$\frac{1}{K}\cdot e^{T\cdot h_{\Gamma_n}} \leq \Gamma_n x_n \cap \overline{B}(x_n,T) \leq K\cdot e^{T\cdot h_{\Gamma_n}}.$$
\end{theorem}
In literature the behaviour of the critical exponent under another kind of convergence, the algebraic convergence, was previously studied by Bishop and Jones in case of hyperbolic manifolds (\cite{BJ97}) and in a more general setting by Paulin in \cite{Pau97}. The definition of algebraic convergence, as well as the notations below, are recalled in Section \ref{subsec-algebraic}. They proved:
\begin{theorem-no}[\cite{BJ97} and \cite{Pau97}, BJP-Theorem]
	\label{theorem-BJP}
	Let $X$ be a geodesic, $\delta$-hyperbolic metric space such that for each $x,y\in X$ there is a geodesic ray issuing from $x$ passing at distance $\leq \delta$ from $y$ and let $G$ be a finitely generated group. Let $\varphi_n,\varphi_\infty\colon G \to \textup{Isom}(X)$ be homomorphisms. If $\varphi_n(G)$ converges algebraically to $\varphi_\infty(G)$, if $\varphi_n(G)$, $\varphi_\infty(G)$ are discrete and if $\varphi_\infty(G)$ has no global fixed point at infinity, then $h_{\varphi_\infty(G)} \leq \liminf_{n\to+\infty}h_{\varphi_n(G)}$.
\end{theorem-no}
\noindent We point out here the main differences and analogies between this statement and Theorem \ref{theorem-main}.(ii): 
\begin{itemize}
	\item[-] in BJP-Theorem the isomorphism type of the group $G$ is fixed. A posteriori this is not restrictive: under the assumptions of Theorem \ref{theorem-main} the isomorphism type of the groups $\Gamma_n$ is eventually constant (Corollary \ref{cor-isomorphic-constant}). The proof of Theorem \ref{theorem-main}.(i) does not use this property.
	\item[-] The algebraic limit is a priori different to the equivariant pointed Gromov-Hausdorff limit (see Example \ref{ex-algebraic-GH}). However if the spaces $X_n$ are all isometric and satisfy the assumptions of Theorem \ref{theorem-main} then the two limits coincide (Theorem \ref{theo-ultralimit-to-algebraic}). 
	\item[-] In Theorem \ref{theorem-main} the spaces $X_n$ can be pairwise non-isometric. In that case the notion of algebraic convergence cannot be defined. This is a posteriori the main difference between algebraic convergence and equivariant pointed Gromov-Hausdorff convergence on the class $\mathcal{M}(\delta,D)$.
\end{itemize}

\section{Preliminaries on metric spaces}
\label{sec-preliminaries}
Throughout the paper $X$ will denote a metric space and $d$ will denote the metric on $X$. The open (resp.closed) ball of radius $r$ and center $x$ is denoted by $B(x,r)$ (resp. $\overline{B}(x,r)$). 
A geodesic segment is an isometry $\gamma\colon I \to X$ where $I=[a,b]$ is a a bounded interval of $\mathbb{R}$. The points $\gamma(a), \gamma(b)$ are called the endpoints of $\gamma$. A metric space $X$ is said geodesic if for all couple of points $x,y\in X$ there exists a geodesic segment whose endpoints are $x$ and $y$. We will denote any geodesic segment between two points $x$ and $y$, with an abuse of notation, as $[x,y]$. A geodesic ray is an isometry $\gamma\colon[0,+\infty)\to X$ while a geodesic line is an isometry $\gamma\colon \mathbb{R}\to X$. \vspace{2mm}

\noindent The group of isometries of a proper (i.e. closed balls are compact) metric space $X$ is denoted by Isom$(X)$ and it is endowed with the topology of uniform convergence on compact subsets of $X$. \\
If $\Gamma$ is a subgroup of Isom$(X)$ we define $\Sigma_R(\Gamma, x) := \lbrace g\in \Gamma \text{ s.t. } d(x,gx)\leq R\rbrace$ and $\Gamma_R(x) := \langle \Sigma_R(\Gamma, x) \rangle$, for every $x\in X$ and $R\geq 0$. When the context is clear we simply write $\Sigma_R(x)$. A subgroup $\Gamma$ is said to be discrete if equivalently:
\begin{itemize}
	\item[(i)] it is a discrete subspace of Isom$(X)$;
	\item[(ii)] $\#\Sigma_R(x) < +\infty$ for all $x\in X$ and all $R\geq 0$. 
\end{itemize}
The \emph{systole} of $\Gamma$ at $x\in X$ is the quantity 
$$\text{sys}(\Gamma,x) := \inf\lbrace d(x,gx) \text{ s.t. } g\in \Gamma\setminus\lbrace \text{id}\rbrace\rbrace,$$
while the \emph{global systole} of $\Gamma$ is $\text{sys}(\Gamma,X) := \inf_{x\in X}\text{sys}(\Gamma,x).$
If any non-trivial isometry of $\Gamma$ has no fixed points then the systole at a point is always strictly positive by discreteness.
%
%
%
%
%

\section{Convergence of group actions}
\label{sec-convergence}
First of all we recall the definition and the properties of the equivariant pointed Gromov-Hausdorff convergence. After that we compare this notion with ultralimit convergence. 
\subsection{Equivariant pointed Gromov-Hausdorff convergence}
We consider triples $(X,x,\Gamma)$ where $(X,x)$ is a pointed, proper metric space and $\Gamma< \text{Isom}(X)$. The following definitions are due to Fukaya, \cite{Fuk86}.
\begin{defin}
	Let $(X,x,\Gamma), (Y,y,\Lambda)$ be two triples as above and $\varepsilon > 0$. An equivariant $\varepsilon$-approximation between them is a triple $(f,\varphi,\psi)$ where 
	\begin{itemize}
		\item $f \colon B(x, \frac{1}{\varepsilon}) \to B(y,\frac{1}{\varepsilon})$ is a map such that
		\begin{itemize}
			\item[-] $f(x)=y$;
			\item[-] $\vert d(f(x_1), f(x_2)) - d(x_1, x_2)\vert <\varepsilon$ for every $x_1,x_2\in B(x,\frac{1}{\varepsilon})$;
			\item[-] for every $y_1\in B(y, \frac{1}{\varepsilon})$ there exists $x_1\in B(x,\frac{1}{\varepsilon})$ such that \linebreak $d(f(x_1),y_1) < \varepsilon$;
		\end{itemize}
		\item $\varphi\colon \Sigma_{\frac{1}{\varepsilon}}(\Gamma, x) \to \Sigma_{\frac{1}{\varepsilon}}(\Lambda, y)$ is a map satisfying $d(f(gx_1), \varphi(g) f(x_1)) <\varepsilon$ for every $g\in \Sigma_{\frac{1}{\varepsilon}}(\Gamma, x)$ and every $x_1\in B(x, \frac{1}{\varepsilon})$ such that also $gx_1 \in B(x, \frac{1}{\varepsilon})$;
		\item $\psi\colon \Sigma_{\frac{1}{\varepsilon}}(\Lambda, y) \to \Sigma_{\frac{1}{\varepsilon}}(\Gamma, x)$ is a map satisfying $d(f(\psi(g)x_1), g f(x_1)) <\varepsilon$ for every $g\in \Sigma_{\frac{1}{\varepsilon}}(\Lambda, y)$ and every $x_1\in B(x, \frac{1}{\varepsilon})$ such that $\psi(g)x_1 \in B(x, \frac{1}{\varepsilon})$.
	\end{itemize}
\end{defin} 
\begin{defin}
	A sequence of triples $(X_n,x_n,\Gamma_n)$ is said to converge in the equivariant pointed Gromov-Hausdorff sense to a triple $(X,x,\Gamma)$ if for every $\varepsilon > 0$ there exists $n_\varepsilon \geq 0$ such that if $n\geq n_\varepsilon$ then there exists an equivariant $\varepsilon$-approximation between $(X_n,x_n,\Gamma_n)$ and $(X,x,\Gamma)$.	One of these equivariant $\varepsilon$-approximations will be denoted by $(f_n,\varphi_n,\psi_n)$. \\
	In this case we will write $(X_n,x_n,\Gamma_n) \underset{\text{eq-pGH}}{\longrightarrow} (X,x,\Gamma)$.
\end{defin}
\begin{obs} A few observations:
	\begin{itemize}
		\item If $(X_n,x_n,\Gamma_n) \underset{\textup{eq-pGH}}{\longrightarrow} (X,x,\Gamma)$ then $(X_n,x_n)$ converges in the classical pointed Gromov-Hausdorff sense to $(X,x)$. We denote this convergence by $(X_n,x_n) \underset{\textup{pGH}}{\longrightarrow} (X,x)$.
		\item In the definition the limit space $X$ is assumed to be proper. This is not restrictive as we see in a moment. If $(X_n,x_n,\Gamma_n) \underset{\textup{eq-pGH}}{\longrightarrow} (X,x,\Gamma)$ we denote by $\hat{X}$ the completion of $X$. Any isometry of $X$ defines uniquely an isometry of $\hat{X}$, so there is a well defined group of isometries $\hat{\Gamma}$ of $\hat{X}$ associated to $\Gamma$. It follows from the definition that $(X_n,x_n,\Gamma_n) \underset{\textup{eq-pGH}}{\longrightarrow} (\hat{X},\hat{x},\hat{\Gamma})$ too.		
		Moreover if a sequence of proper metric spaces converges in the pointed Gromov-Hausdorff sense to a complete metric space then the limit is proper by Corollary 3.10 of \cite{Her16}. 
	\end{itemize}
\end{obs}
We recall that Isom$(X)$ is endowed with the topology of uniform convergence on compact subsets of $X$, when $X$ is a proper space. It is classically known (see also the remark above) that the pointed Gromov-Hausdorff limit of a sequence of metric spaces is unique up to pointed isometry when we restrict to the class of complete (and therefore proper) spaces. In order to obtain uniqueness of the equivariant pointed Gromov-Hausdorff limit we need to restrict to groups that are closed in the isometry group of the limit space.
\begin{prop}[Proposition 1.5 of \cite{Fuk86}]
	Suppose $(X_n,x_n,\Gamma_n) \underset{\textup{eq-pGH}}{\longrightarrow} (X,x,\Gamma)$ and $(X_n,x_n,\Gamma_n) \underset{\textup{eq-pGH}}{\longrightarrow} (Y,y,\Lambda)$, where $X,Y$ are proper and $\Gamma,\Lambda$ are closed subgroups of \textup{Isom}$(X)$, \textup{Isom}$(Y)$ respectively.
	Then there exists an isometry $F\colon X \to Y$ such that
	\begin{itemize}
		\item[-] $F(x)=y$;
		\item[-] $F_*\colon \textup{Isom}(X) \to \textup{Isom}(Y)$ defined by $F_*(g) = F\circ g \circ F^{-1}$ is an isomorphism between $\Gamma$ and $\Lambda$.
	\end{itemize}
\end{prop}
\begin{defin}
	Two triples $(X,x,\Gamma)$ and $(Y,y,\Lambda)$ are said equivariantly isometric if there exists an isometry $F\colon X \to Y$ satisfying the thesis of the previous proposition. In this case $F$ is called an equivariant isometry and we write $(X,x,\Gamma)\cong(Y,y,\Lambda)$.
\end{defin}
\noindent \emph{From now on we will consider the equivariant pointed Gromov-Hausdorff convergence only restricted to triples $(X,x,\Gamma)$ where $(X,x)$ is a pointed proper metric space and $\Gamma$ is a closed subgroup of \textup{Isom}$(X)$.}\vspace{2mm}

\noindent This condition is not restrictive.
\begin{lemma}
	If $(X_n,x_n,\Gamma_n) \underset{\textup{eq-pGH}}{\longrightarrow} (X,x,\Gamma)$ then $(X_n,x_n,\Gamma_n) \underset{\textup{eq-pGH}}{\longrightarrow} (X,x,\bar\Gamma)$, where $\bar\Gamma$ is the closure of $\Gamma$ in \textup{Isom}$(X)$.
\end{lemma}
\begin{proof}
	By definition for every $\varepsilon > 0$ there is $n_\varepsilon \geq 0$ such that for every $n\geq n_\varepsilon$ there is an equivariant $\frac{\varepsilon}{2}$-approximation $(f_n,\varphi_n,\psi_n)$ between $(X_n,x_n,\Gamma_n)$ and $(X,x,\Gamma)$. We want to define an equivariant $2\varepsilon$-approximation $(f_n,\varphi_n,\bar\psi_n)$ between $(X_n,x_n,\Gamma_n)$ and $(X,x,\bar\Gamma)$.\\ For every $g\in \Sigma_\frac{1}{\varepsilon}(\bar\Gamma,x)$ there is a sequence of isometries $g_k\in \Gamma$ such that $g_k \to g$ uniformly on compact subsets of $X$. In particular for every $\delta > 0$ there exists $k_\delta \geq 0$ such that if $k\geq k_\delta$ then $d(g_k(y), g(y)) \leq \delta$ for every $y\in \overline{B}(x, \frac{2}{\varepsilon})$.
	Choosing $\delta$ small enough we have $g_k \in \Sigma_{\frac{2}{\varepsilon}}(\Gamma, x)$ for $k \geq k_\delta$. We define $\bar\psi_n(g) := \psi_n(g_{k_\delta})$. Observe that for every $y_n\in B(x_n,\frac{1}{\varepsilon})$ we have
	\begin{equation*}
		\begin{aligned}
			d(f_n(\bar\psi(g)y_n), gf_n(y_n)) &= d(f_n(\psi_n(g_{k_\delta})), gf_n(y_n)) \\
			&\leq d(f_n(\psi_n(g_{k_\delta})), g_{k_\delta}f_n(y_n)) + d( g_{k_\delta}f_n(y_n), gf_n(y_n)) \\
			&\leq \varepsilon + \delta,
		\end{aligned}
	\end{equation*}
	where the last inequality follows since $f_n(y_n)\in \overline{B}(x, \frac{2}{\varepsilon})$. Taking $\delta \leq \varepsilon$ we conclude that $(f_n,\varphi_n, \bar\psi_n)$ is the desired equivariant $2\varepsilon$-approximation and it is defined for all $n\geq n_\varepsilon$.
\end{proof}

\subsection{Ultralimit of groups}
For more detailed notions on ultralimits we refer to \cite{DK18} and \cite{CavS20}. A non-principal ultrafilter $\omega$ is a finitely additive measure on $\mathbb{N}$ such that $\omega(A) \in \lbrace 0,1 \rbrace$ for every $A\subseteq \mathbb{N}$ and $\omega(A)=0$ for every finite subset of $\mathbb{N}$. Accordingly we write $\omega$-a.s. and for $\omega$-a.e.$(n)$ in the usual measure theoretic sense. \\
Given a bounded sequence $(a_n)$ of real numbers and a non-principal ultrafilter $\omega$ there exists a unique $a\in \mathbb{R}$ such that for every $\varepsilon > 0$ the set $\lbrace n \in \mathbb{N} \text{ s.t. } \vert a_n - a \vert < \varepsilon\rbrace$ has $\omega$-measure $1$, see for instance \cite{DK18}, Lemma 10.25. The real number $a$ is called the ultralimit of the sequence $a_n$ and it is denoted by $\omega$-$\lim a_n$.\\
Given a sequence of pointed metric spaces $(X_n, x_n)$ we denote by $(X_\omega, x_\omega)$ the ultralimit pointed metric space. It is the set of sequences $(y_n)$, where $y_n\in X_n$ for every $n$, for which there exists $M\in\mathbb{R}$ such that $d(x_n,y_n)\leq M$ for $\omega$-a.e.$(n)$, modulo the relation $(y_n)\sim (y_n')$ if and only if $\omega$-$\lim d(y_n,y_n') = 0$. The point of $X_\omega$ defined by the class of the sequence $(y_n)$ is denoted by $y_\omega = \omega$-$\lim y_n$.
The formula $d(\omega$-$\lim y_n, \omega$-$\lim y_n') = \omega$-$\lim d(y_n,y_n')$ defines a metric on $X_\omega$ which is called the ultralimit distance on $X_\omega$.\\
A sequence of isometries $g_n \in \text{Isom}(X_n)$ is admissible if there exists $M\geq 0$ such that $d(g_n x_n, x_n) \leq M$ $\omega$-a.s. Any such sequence defines an isometry $g_\omega = \omega$-$\lim g_n$ of $X_\omega$ by the formula $g_\omega y_\omega = \omega$-$\lim g_ny_n$, Lemma 10.48 of \cite{DK18}.
Given a sequence of groups of isometries $\Gamma_n$ of $X_n$ we set
$$\Gamma_\omega = \lbrace \omega\text{-}\lim g_n \text{ s.t. } g_n \in \Gamma_n \text{ for } \omega\text{-a.e.}(n)\rbrace.$$
In particular the elements of $\Gamma_\omega$ are ultralimits of admissible sequences. 
\begin{lemma}
	\label{composition}
	The composition of admissible sequences of isometries is an admissible sequence of isometries and the limit of the composition is the composition of the limits.
\end{lemma}
\noindent  (Indeed, if  $g_\omega = \omega$-$\lim g_n$, $h_\omega = \omega$-$\lim h_n$ belong to $\Gamma_\omega$ then  their composition belong to $\Gamma_\omega$, as $\omega$-$\lim d(g_n h_n \cdot x_n, x_n)\leq \omega$-$\lim d(g_n h_n \cdot x_n, g_n \cdot x_n) + \omega$-$\lim d(g_n \cdot x_n, x_n) < + \infty$).

\noindent Analogously one proves that $(\text{id}_n)$ belongs to $\Gamma_\omega$ and defines the identity map of $X_\omega$, and that  if $g_\omega = \omega$-$\lim g_n$ belongs to $\Gamma_\omega$ then also the sequence $(g_n^{-1})$ defines an element of $\Gamma_\omega$, which is the inverse of $g_\omega$.

\noindent So we have a well defined composition law on $\Gamma_\omega$, that is for $ g_\omega   = \omega$-$\lim g_n$ and  $ h_\omega = \omega$-$\lim h_n$ we set $g_\omega \, \circ\, h_\omega = \omega\text{-}\lim(g_n \circ h_n).$
With this operation $\Gamma_\omega$ is a group of isometries of $X_\omega$ and we call it {\em the ultralimit group} of the sequence of groups $\Gamma_n$. \\
%
%
%
The ultralimit space $X_\omega$ may be not proper in general, even if $X_n$ is proper for every $n$. When $X_\omega$ is proper then $\Gamma_\omega$ is closed with respect to the uniform convergence on compact subsets.
\begin{prop}
	\label{ultralimit-closed}
	Let $(X_n,x_n,\Gamma_n)$ be a sequence of proper metric spaces and $\omega$ be a non-principal ultrafilter. If $X_\omega$ is proper then $\Gamma_\omega$ is a closed subgroup of \textup{Isom}$(X_\omega)$.
\end{prop}
We remark that the proof is analogous to Corollary 10.64 of \cite{DK18}.
\begin{proof}
	Let $(g_\omega^k)^{k\in \mathbb{N}}$ be a sequence of isometries of $\Gamma_\omega$ converging to an isometry $g^\infty$ of $X_\omega$ with respect to the uniform convergence on compact subsets.
	We want to show that $g^\infty$ coincides with the ultralimit of some sequence of admissible isometries $g_n^\infty$ of $X_n$.\\
	First of all we can extract a subsequence, denoted again by $(g_\omega^k)$, satisfying 
	$d(g_\omega^k y_\omega, g^{k+1}_\omega y_\omega) \leq \frac{1}{2^k}$ for all $y_\omega\in \overline{B}(x_\omega, k)$, for all $k\in \mathbb{N}$.
	Now, for every fixed $k\in \mathbb{N}$ let $S^k = \lbrace y_\omega^1,\ldots,y_\omega^{N_k}\rbrace$ be a $\frac{1}{2^{k+2}}$-dense subset of $B(x_\omega, k)$. It is finite since $X_\omega$ is proper. If $y_\omega^i = \omega$-$\lim y_n^i$ for $i=1,\ldots,N_k$, it is clear that the set $S^k_n = \lbrace y_n^1,\ldots, y_n^{N_k}\rbrace$ is a $\frac{1}{2^{k+1}}$-dense subset of $B(x_n, k)$ for $\omega$-a.e.$(n)$.
	For every $k\in \mathbb{N}$ we define the set
	$$A_k = \left\lbrace n \in \mathbb{N} \text{ s.t. } d(g_n^jy_n^i, g_n^{j+1}y_n^i) \leq \frac{1}{2^j} \text{ for all } 1\leq j \leq k,\,\,\, i = 1,\ldots, N_j\right\rbrace.$$
	By definition $\omega(A_k)=1$ and $A_{k+1}\subseteq A_{k}$ for every $k\in \mathbb{N}$. We set $B=\bigcap_{k\in \mathbb{N}} A_k$.
	There are two cases:\\
	\textbf{Case 1: $\omega(B)=1$.} In this case
	$d(g_n^ky_n^i, g_n^{k+1}y_n^i) \leq \frac{1}{2^k}$
	for all $i=1,\ldots, N_k$ and all $k\in \mathbb{N}$, $\omega$-a.s. Then $\omega$-a.s. we have $d(g_n^ky_n, g_n^{k+1}y_n) \leq \frac{1}{2^{k-1}}$ for every $k\in \mathbb{N}$ and every $y_n\in B(x_n,k)$. We set $g_n^\infty := g_n^n \in \Gamma_n$. This sequence of isometries is admissible and we denote its ultralimit by $g_\omega^\infty \in \Gamma_\omega$. Now we fix $y_\omega = \omega$-$\lim y_n \in X_\omega$. By definition $d(x_n,y_n)\leq M$ for $\omega$-a.e.$(n)$. We have
	\begin{equation*}
		\begin{aligned}
			d(g^\infty y_\omega, g^\infty_\omega y_\omega) \leq d(g^k_\omega y_\omega, g^\infty_\omega y_\omega) + \frac{1}{2^{k}} &= \omega\text{-}\lim d(g_n^ky_n, g_n^n y_n) + \frac{1}{2^k} \\
			&\leq \sum_{j=0}^{n-k-1}d(g_n^{k+j}y_n, g_n^{k+j+1}y_n) + \frac{1}{2^k} \\
			&\leq \sum_{j=k-1}^\infty \frac{1}{2^j} + \frac{1}{2^k} \leq \frac{1}{2^{k-2}} + \frac{1}{2^k} \leq \frac{1}{2^{k-3}}
		\end{aligned}
	\end{equation*}
	for all fixed $k\geq M$. We conclude that $g^\infty y_\omega = g^\infty_\omega y_\omega$. By the arbitrariness of $y_\omega$ we get $g^\infty = g^\infty_\omega$ that belongs to $\Gamma_\omega$.\\
	\textbf{Case 2: $\omega(B)=0$.} Since $A_1 = \bigsqcup_{j=1}^\infty (A_j\setminus A_{j+1})  \sqcup B$ and since $\omega(A_1) = 1$ then $\omega(\bigsqcup_{j=1}^\infty (A_j\setminus A_{j+1})) = 1$. We set $C = \bigsqcup_{j=1}^\infty (A_j\setminus A_{j+1})$. For all $n \in C$ we set $g_n^\infty := g_n^{j(n)}$, where $j(n)$ is the unique $j\geq 1$ such that $n\in A_j\setminus A_{j+1}$. We claim that the corresponding ultralimit isometry $g_\omega^\infty \in \Gamma_\omega$ equals $g^\infty$. Indeed let $y_\omega= \omega$-$\lim y_n \in X_\omega$. For every fixed $k\in \mathbb{N}$ consider the set $C_k = \bigsqcup_{j=k}^\infty (A_j\setminus A_{j+1})$. Each of the sets $A_j\setminus A_{j+1}$ has $\omega$-measure $0$, so $\omega(C_k) = 1$ for every fixed $k$. Let $n\in C_k \subset C$. By definition $j(n)\geq k$ since $n\in C_k$. Since $A_{j(n)}\subseteq A_{j(n)-1} \subseteq \ldots \subseteq A_k$ we have
	\begin{equation*}
		\begin{aligned}
			d(g_n^{j(n)}y_n, g_n^k y_n) \leq 2\cdot\frac{1}{2^{k+2}} + \sum_{m = k}^{j(n)} \frac{1}{2^m} \leq \frac{1}{2^{k+1}}\cdot \frac{1}{2^{k-1}} \leq \frac{1}{2^{k-2}}
		\end{aligned}
	\end{equation*}
	as soon as $k> d(y_\omega, x_\omega)$ and $n\in C_k$.
	The set $C_k$ has $\omega$-measure $1$, so we conclude that $d(g_\omega^\infty y_\omega, g^\infty y_\omega) \leq \frac{1}{2^{k-2}}$. By the arbitrariness of $k$ and $y_\omega$ we finally get $g^\infty = g_\omega^\infty \in \Gamma_\omega$.	
\end{proof}


\begin{lemma}
	\label{lemma-ultralimit-constant}
	Let $(X,x)$ be a proper metric space and $\Gamma\subseteq \textup{Isom}(X)$ be a closed subgroup. Then the ultralimit of the constant sequence $(X,x,\Gamma)$ is naturally equivariantly isometric to $(X,x,\Gamma)$ for every non-principal ultrafilter.
\end{lemma}
\begin{proof}
	Let $(X_n,x_n,\Gamma_n) = (X,x,\Gamma)$ for every $n$ and let $\omega$ be a non-principal ultrafilter. By \cite{CavS20}, Proposition A.3 the map $\iota\colon (X,x) \to (X_\omega,x_\omega)$ that sends each point $y$ to the ultralimit point corresponding to the constant sequence $y_n=y$ is an isometry sending $x$ to $x_\omega$. This means that each point $y_\omega$ of $X_\omega$ can be written as $y_\omega = \omega$-$\lim y$ for some $y\in X$. We need to show that $\iota_*\Gamma = \Gamma_\omega$. We have $\iota_*g(\omega$-$\lim y) = \omega$-$\lim gy$
	for every $g\in\Gamma$ and $y_\omega = \omega$-$\lim y \in X_\omega$, i.e. $\iota_*g$ coincides with the ultralimit of the constant sequence $g_n = g$. In particular $\iota_*\Gamma\subseteq \Gamma_\omega$. Now we take $g_\omega = \omega$-$\lim g_n \in \Gamma_\omega$, where $g_n\in \Gamma$ is an admissible sequence, i.e. $d(x,g_nx)\leq M$ for some $M$. The set $\Sigma_M(\Gamma, x)$ is compact by Ascoli-Arzela's Theorem (\cite{Kel17}, Chapter 7, Theorem 17) since $\Gamma$ is closed, so by Lemma 10.25 of \cite{DK18} there exists $g\in \Sigma_M(\Gamma, x)$ such that for every $\varepsilon > 0$ and $R\geq 0$ the set
	$$\lbrace n \in \mathbb{N} \text{ s.t. } d(g_ny,gy)<\varepsilon \text{ for all } y \in \overline{B}(x,R) \rbrace$$
	belongs to $\omega$. It is clear that the constant sequence $g$ defines the ultralimit isometry $g_\omega$, i.e. $\iota_*(g) = g_\omega$. Since $g\in \Gamma$ we conclude that $\iota_*\Gamma = \Gamma_\omega$.
\end{proof}

\subsection{Comparison between the two convergences}
We compare now the ultralimit convergence to the equivariant pointed Gromov-Hausdorff convergence. The analogous of the next results for the classical pointed Gromov-Hausdorff convergence can be found for instance in \cite{Jan17}.

\begin{prop}
	\label{GH---ultralimit}
	Suppose $(X_n,x_n,\Gamma_n) \underset{\textup{eq-pGH}}{\longrightarrow} (X,x,\Gamma)$ and call $(f_n,\varphi_n,\psi_n)$ some corresponding equivariant approximations. Let $\omega$ be a non-principal ultrafilter and let $(X_\omega, x_\omega, \Gamma_\omega)$ be the ultralimit triple. Then the map
	$$F \colon (X_\omega, x_\omega, \Gamma_\omega) \to (X,x,\Gamma)$$
	defined by sending $y_\omega = \omega$-$\lim y_n \in X_\omega$ to $\iota^{-1}(\omega$-$\lim f_n(y_n))$ is a well defined equivariant isometry. Here $\iota$ is the natural equivariant isometry of Lemma \ref{lemma-ultralimit-constant} between $(X,x,\Gamma)$ and the ultralimit of its constant sequence.
\end{prop}
\begin{proof}
	We divide the proof in steps.\\
	\textbf{Good definition.} Given a point $y_\omega = \omega$-$\lim y_n \in X_\omega$ by definition there exists $M\geq 0$ such that $d(x_n,y_n)\leq M$ $\omega$.a.s. For $n$ big enough the map $f_n$ is defined on $y_n$ and it satisfies $d(f_n(y_n), f_n(x_n)) \leq M + 1$ and $f_n(x_n)=x$. Then the sequence $(f_n(y_n))$ is $\omega$-a.s. bounded and $\iota^{-1}(\omega$-$\lim f_n(y_n))$ is a well defined point of $X$.\\
	Now suppose $(y_n')$ is another sequence such that $\omega$-$\lim d(y_n,y_n') = 0$. For every $\varepsilon > 0$ we have $d(y_n,y_n')<\varepsilon$ $\omega$-a.s. Moreover, arguing as before, $d(f_n(y_n), f_n(y_n')) \leq d(y_n,y_n') + \varepsilon \leq 2\varepsilon$ $\omega$-a.s. By the arbitrariness of $\varepsilon > 0$ we get $d(\omega$-$\lim f_n(y_n), \omega$-$\lim f_n(y_n')) = 0$. In particular $F$ is well defined.\\
	\textbf{Isometric embedding.} We fix $y_\omega = \omega$-$\lim y_n, z_\omega = \omega$-$\lim z_n \in X_\omega$ and $\varepsilon > 0$. As usual all the conditions
	\begin{equation*}
		\begin{aligned}
			&d(\iota^{-1}(\omega\text{-}\lim f_n(y_n)), f_n(y_n)) < \varepsilon, \qquad \qquad d(\iota^{-1}(\omega\text{-}\lim f_n(z_n)), f_n(z_n)) < \varepsilon,\\
			&\vert d(y_\omega, z_\omega) - d(y_n,z_n)\vert <\varepsilon, \qquad \qquad \vert d(f_n(y_n), f_n(z_n)) - d(y_n,z_n) \vert < \varepsilon.
		\end{aligned}
	\end{equation*}
	hold $\omega$-a.s. Therefore 
	$$\vert d(F(y_\omega), F(z_\omega)) - d(y_\omega, z_\omega) \vert < 4\varepsilon.$$
	By the arbitrariness of $\varepsilon$ we conclude that $F$ is an isometric embedding.\\
	\textbf{Surjectivity.} We fix $y\in X$, $\varepsilon > 0$ and we set $L:= d(x,y)$. By definition there exists $y_n \in X_n$ such that $d(f_n(y_n),y)<\varepsilon$ $\omega$-a.s. The sequence $y_n$ is clearly admissible and defines a point $y_\omega = \omega$-$\lim y_n$ of $X_\omega$. Since $d(f_n(y_n), y) < \varepsilon$ $\omega$-a.s. then $F(y_\omega) = \omega$-$\lim f_n(y_n)$ satisfies $d(F(y_\omega), y) < 2\varepsilon$. This shows that $y$ belongs to the closure of $F(X_\omega)$. Every ultralimit space is a complete metric space (\cite{DK18}, Corollary 10.64), so $F$ is a closed map. Indeed if $F(y_\omega^k)$ is a convergent sequence, then it is Cauchy. Since $F$ is an isometric embedding then the sequence $(y_\omega^k)$ is Cauchy and therefore it converges. By continuity of $F$ we conclude that the limit point of the sequence $F(y_\omega^k)$ belongs to the image of $F$. Hence $F(X_\omega)$ is closed and $y\in F(X_\omega)$, showing that $F$ is surjective.\\
	\textbf{$F$ is equivariant.} It is clear that $F(x_\omega) = \iota^{-1}(\omega$-$\lim f_n(x_n)) = x$. It remains to show that $F_*(\Gamma_\omega) = \Gamma$.
	We take $g_\omega = \omega$-$\lim g_n \in \Gamma_\omega$. Then $F_*(g_\omega)$ acts on the point $y$ of $X$ as 
	$F_* g_\omega (y) = F\circ g_\omega \circ F^{-1}(y)$. Clearly $F^{-1}(y) = \omega$-$\lim y_n$, where $y_n$ is any sequence such that $\omega$-$\lim f_n(y_n) = y$. So $g_\omega \circ F^{-1} y = \omega$-$\lim g_n y_n$ by definition of $g_\omega$. Finally $F\circ g_\omega \circ F^{-1} = \iota^{-1}(\omega$-$\lim f_n(g_ny_n))$. Now we consider the isometries $\varphi_n(g_n)\in \Gamma$: they are defined for $n$ big enough since $g_\omega$ displaces $x_\omega$ of some finite quantity. We define the isometry $\iota_*^{-1}(\omega$-$\lim \varphi_n(g_n))$: it is an isometry of $X$, which is proper, and it belongs to $\Gamma$ by Lemma \ref{lemma-ultralimit-constant}. We have
	$$d(\iota_*^{-1}\omega\text{-}\lim \varphi_n(g_n)(y), F\circ g_\omega \circ F^{-1}(y)) = \omega\text{-}\lim d(\varphi_ng_n(y), f_n(g_ny_n))$$
	for every $y\in X$.
	Moreover
	\begin{equation*}
		\begin{aligned}
			d(\varphi_ng_n(y), f_n(g_ny_n)) &\leq d(\varphi_ng_n(y), \varphi_ng_n(f_n(y_n))) + d(\varphi_ng_n(f_n(y_n)), f_n(g_ny_n)) \\
			&\leq 2\varepsilon
		\end{aligned}
	\end{equation*}
	if $n$ is big enough. This means that  $F_* g_\omega = \iota^{-1} (\omega$-$\lim \varphi_n g_n) \in \Gamma$, so $F_*\Gamma_\omega \subseteq \Gamma$.
	Now we take $g\in \Gamma$ and we consider the isometries $\psi_n g \in \Gamma_n$ that are defined for $n$ big enough. The sequence $(\psi_n g) $ is admissible and therefore it defines a limit isometry $g_\omega$. For all $y\in X$ we have
	$F_*(g_\omega)(y) = Fg_\omega(y_\omega)$, where $y_\omega = \omega$-$\lim y_n$ and $y_n$ is a sequence such that $\iota^{-1}(\omega$-$\lim f_n(y_n)) = y$. Then $F_*(g_\omega)(y) = \iota^{-1}(\omega$-$\lim f_n(\psi_n(g)y_n))$. Once again $\omega$-a.s. we have
	\begin{equation*}
		\begin{aligned}
			d(F_*(g_\omega)(y), gy) &\leq d(\iota^{-1}(\omega\text{-}\lim f_n(\psi_n(g)y_n)), gf_n(y_n)) +\varepsilon\\
			&\leq d(f_n(\psi_n(g)y_n), gf_n(y_n)) + 2\varepsilon.
		\end{aligned}
	\end{equation*}
	We conclude that $F_*g_\omega = g$, so $F_*\Gamma_\omega = \Gamma$.
\end{proof}

\begin{prop}
	Let $(X_n,x_n,\Gamma_n)$ be a sequence of triples, $\omega$ be a non-principal ultrafilter and $(X_\omega, x_\omega, \Gamma_\omega)$ be the ultralimit triple. If $X_\omega$ is proper then there exists a subsequence $\lbrace n_k \rbrace \subseteq \mathbb{N}$ such that $(X_{n_k},x_{n_k},\Gamma_{n_k}) \underset{\textup{eq-pGH}}{\longrightarrow} (X_\omega,x_\omega,\Gamma_\omega)$.
\end{prop}
\begin{obs}
	Notice that the subsequence $\lbrace n_k \rbrace$ may not belong to $\omega$.
\end{obs}
\begin{proof}
	We fix $\varepsilon > 0$. Since $X_\omega$ is proper we can select a $\frac{\varepsilon}{7}$-net $S^\varepsilon = \lbrace x_\omega = y_\omega^1,\ldots,y_\omega^{N_\varepsilon}\rbrace$ of $B(x_\omega, \frac{1}{\varepsilon})$, where $y_\omega^i = \omega$-$\lim y_n^i$ for $i=1,\ldots,N_\varepsilon$.
	Moreover $\Gamma_\omega$ is closed by Proposition \ref{ultralimit-closed}, so $\Gamma_{\omega,\frac{1}{\varepsilon}}(x_\omega)$ is relatively compact by Ascoli-Arzela's Theorem (\cite{Kel17}, Chapter 7, Theorem 17). Therefore we can find a finite subset $g_\omega^1,\ldots,g_\omega^{K_\varepsilon} \in \Gamma_{\omega, \frac{1}{\varepsilon}}(x_\omega)$, $g_\omega^i = \omega$-$\lim g_n^i$, with the property that for every $g_\omega\in \Gamma_{\omega, \frac{1}{\varepsilon}}(x_\omega)$ there exists $1\leq i \leq K_\varepsilon$ such that $d(g_\omega y_\omega, g_\omega^i y_\omega) \leq \frac{\varepsilon}{7}$ for all $y_\omega \in B(x_\omega, \frac{1}{\varepsilon})$.\\
	Now $\omega$-a.s. the following finite set of conditions hold:
	\begin{itemize}
		\item[-] $\vert d(y_\omega^i, y_\omega^j) - d(y_n^i, y_n^j) \vert \leq \frac{\varepsilon}{7}$ for all $i,j\in \lbrace 1,\ldots, N_\varepsilon\rbrace$;
		\item[-] the set $S_n^\varepsilon = \lbrace y_n^1,\ldots, y_n^{N_\varepsilon}\rbrace$ is a $\frac{2}{7}\varepsilon$-net of $B(x_n,\frac{1}{\varepsilon})$;
		\item[-] $g_n^i \in \Gamma_{n,\frac{1}{\varepsilon}}(x_n)$ for every $i=1,\ldots,K_\varepsilon$;
		\item[-] $\vert d(g_\omega^iy_\omega^j, y_\omega^l) - d(g_ny_n^j, y_n^l) \vert \leq \frac{\varepsilon}{7}$ for all $1\leq i \leq K_\varepsilon$ and all $1\leq j,l \leq N_\varepsilon$.
		\item[-] the set $\lbrace g_n^1,\ldots, g_n^{K_\varepsilon}\rbrace$ is a $\frac{2}{7}\varepsilon$-dense subset of $\Gamma_{n,\frac{1}{\varepsilon}}(x_n)$ with respect to the uniform distance.
	\end{itemize}
	For the natural numbers $n$ where these conditions hold we define 
	$$f_n\colon B\left(x_n,\frac{1}{\varepsilon}\right) \to B\left(x_\omega, \frac{1}{\varepsilon}\right)$$
	by sending the point $y_n$ to a point $y_\omega^i$ where $i$ is such that $d(y_n,y_n^i) \leq \frac{2}{7}\varepsilon$.
	For $y_n,z_n\in B(x_n,\frac{1}{\varepsilon})$ we have
	$$\vert d(f_n(y_n), f_n(z_n)) - d(y_n,z_n)\vert = \vert d(y_\omega^{i_1}, y_\omega^{i_2}) - d(y_n,z_n)\vert$$
	for some $i_1,i_2$. But for these indices $n$ we have $\vert d(y_\omega^{i_1}, y_\omega^{i_2}) - d(y_n^{i_1}, y_n^{i_2})\vert \leq \frac{2}{7}\varepsilon$, so we get 
	$$\vert d(f_n(y_n), f_n(z_n)) - d(y_n,z_n)\vert \leq \frac{6}{7}\varepsilon.$$
	Moreover we define 
	$$\psi_n \colon \Gamma_{\omega,\frac{1}{\varepsilon}}(x_\omega) \to \Gamma_{n,\frac{1}{\varepsilon}}(x_n)$$
	by sending $g_\omega$ to $g_n^i$, where $i \in \lbrace 1,\ldots, K_\varepsilon\rbrace$ is such that $d^\infty_{B(x_\omega, 1/\varepsilon)}(g_\omega, g_\omega^i) \leq \frac{\varepsilon}{7}$.\\
	Let $g_\omega \in \Gamma_{\omega,\frac{1}{\varepsilon}}(x_\omega)$, so $\Psi_n(g_\omega)=g_n^i$ as before. Let $y_n \in B(x_n,\frac{1}{\varepsilon})$ such that also $g_n^iy_n\in B(x_n,\frac{1}{\varepsilon})$. Let $j,l\in \lbrace 1,\ldots, N_\varepsilon\rbrace$ be such that $d(y_n,y_n^j) \leq \frac{2}{7}\varepsilon$ and $d(g_n^iy_n, y_n^l)\leq \frac{2}{7}\varepsilon$. 
	By definition $f_n(y_n)=y_\omega^i$, while $f_n(\psi_n(g_\omega)y_n) = y_\omega^l$. We have
	\begin{equation*}
		\begin{aligned}
			d(f_n(\psi_n(g_\omega)y_n), g_\omega f_n(y_n)) &= d(y_\omega^l, g_\omega y_\omega^j)\\
			&\leq d(y_\omega^l, g_\omega^i y_\omega^j) + \frac{2}{7}\varepsilon \\
			&\leq d(y_n^l, g_n^i y_n^j) + \frac{3}{7}\varepsilon \leq \varepsilon.
		\end{aligned}
	\end{equation*}
	Finally we define 
	$$\varphi_n \colon \Gamma_{n,\frac{1}{\varepsilon}}(x_n) \to \Gamma_{\omega, \frac{1}{\varepsilon}}(x_\omega)$$
	as $\varphi_n(g_n)=g_\omega^i$, where $d^\infty_{B(x_n,1/\varepsilon)}(g_n, g_n^i)\leq \frac{2}{7}\varepsilon$.\\
	Let $g_n \in \Gamma_{n,\frac{1}{\varepsilon}}(x_n)$ and let $\varphi_n(g_n)=g_\omega^i$. Now let $y_n\in B(x_n,\frac{1}{\varepsilon})$ such that also $g_ny_n \in B(x_n, \frac{1}{\varepsilon})$. Let $j,l\in \lbrace 1,\ldots, N_\varepsilon\rbrace$ such that $d(y_n, y_n^j)\leq \frac{2}{7}\varepsilon$ and $d(g_ny_n, y_n^l)\leq \frac{2}{7}\varepsilon$, so that $f_n(y_n)=y_\omega^j$ and $f_n(g_ny_n)=y_\omega^l$. Therefore
	\begin{equation*}
		\begin{aligned}
			d(f_n(g_ny_n), \varphi_n(g_n)f_n(y_n)) &= d(y_\omega^l, g_\omega^i y_\omega^j)\\
			&\leq d(y_n^l, g_n y_n^j) + \frac{3}{7}\varepsilon \\
			&\leq d(y_n^l, g_n y_n) + \frac{5}{7}\varepsilon \leq \varepsilon.
		\end{aligned}
	\end{equation*}
	This shows that $\omega$-a.s. we can find an equivariant $\varepsilon$-approximation between $(X_n,x_n,\Gamma_n)$ and $(X_\omega, x_\omega, \Gamma_\omega)$.
	For all integers $k$ we set $\varepsilon = \frac{1}{k}$ and we choose $n_k \in \mathbb{N}$ in the set of indices for which there exists an equivariant $\frac{1}{k}$-approximation as before. The sequence $(X_{n_k}, x_{n_k}, \Gamma_{n_k})$ satisfies the thesis.
\end{proof}

We summarize these properties in the following
\begin{prop}
	\label{prop-GH-ultralimit}
	Let $(X_n, x_n, \Gamma_n)$ be a sequence of triples and let $\omega$ be a non-principal ultrafilter.
	\begin{itemize}
		\item[(i)] If $(X_n,x_n,\Gamma_n) \underset{\textup{eq-pGH}}{\longrightarrow} (X,x,\Gamma)$ then $(X_\omega, x_\omega, \Gamma_\omega) \cong (X,x,\Gamma)$.
		\item[(ii)] If $X_\omega$ is proper then $(X_{n_k},x_{n_k},\Gamma_{n_k}) \underset{\textup{eq-pGH}}{\longrightarrow} (X_\omega,x_\omega,\Gamma_\omega)$ for some subsequence $\lbrace{n_k}\rbrace$.
	\end{itemize}
\end{prop}
\begin{cor}
	\label{cor-ultralimit-unique-limit}
	Let $(X_n, x_n, \Gamma_n)$ be a sequence of triples and suppose that there is a triple $(X,x,\Gamma)$, $X$ proper, such that $(X,x,\Gamma) \cong (X_\omega,x_\omega,\Gamma_\omega)$ for every non-principal ultrafilter $\omega$. Then $(X_n,x_n,\Gamma_n) \underset{\textup{eq-pGH}}{\longrightarrow} (X,x,\Gamma)$.
\end{cor}
\begin{proof}
	The equivariant pointed Gromov-Hausdorff convergence is metrizable (cp. \cite{Fuk86}), so it is enough to show that every subsequence has a subsequence that converges to $(X,x,\Gamma)$. Fix a subsequence $\lbrace n_k \rbrace$. The set $\{n_k\}$ is infinite, then there exists a non-principal ultrafilter $\omega$ containing it for which $\omega(\{n_k\})=1$ (cp. \cite{Jan17}, Lemma 3.2). The ultralimit with respect to $\omega$ of the sequence $(X_{n_k},x_{n_k},\Gamma_{n_k})$ is the same of the sequence $(X_n,x_n,\Gamma_n)$ since $\omega(\{n_k\})=1$. By Proposition \ref{prop-GH-ultralimit} we can extract a further subsequence $\lbrace n_{k_j} \rbrace$ that converges in the equivariant pointed Gromov-Hausdorff sense to $(X,x,\Gamma)$.
\end{proof}

\section{Gromov hyperbolic metric spaces}
\label{sec-Gromov}

We recall briefly the definition and some properties of Gromov-hyperbolic metric spaces. Good references are for instance \cite{BH09} and \cite{CDP90}.

\noindent Let $X$ be a geodesic metric space. Given  three points  $x,y,z \in X$,  the {\em Gromov product} of $y$ and $z$ with respect to $x$  is defined as
\vspace{-3mm}

$$(y,z)_x = \frac{1}{2}\big( d(x,y) + d(x,z) - d(y,z) \big).$$

\noindent The space $X$ is said {\em $\delta$-hyperbolic}, $\delta \geq 0$, if   for every four points $x,y,z,w \in X$   the following {\em 4-points condition} hold:
\begin{equation}\label{hyperbolicity}
	(x,z)_w \geq \min\lbrace (x,y)_w, (y,z)_w \rbrace -  \delta 
\end{equation}

\vspace{-2mm}
\noindent  or, equivalently,
\vspace{-5mm}

\begin{equation}
	\label{four-points-condition}
	d(x,y) + d(z,w) \leq \max \lbrace d(x,z) + d(y,w), d(x,w) + d(y,z) \rbrace + 2\delta. 
\end{equation}

\noindent The space $X$ is   {\em Gromov hyperbolic} if it is $\delta$-hyperbolic for some $\delta \geq 0$. \\
This formulation of $\delta$-hyperbolicity is convenient when interested in taking limits. 
We will also make use of another classical characterization of  $\delta$-hyperbolicity. A {\em geodesic triangle} in $X$ is the union of three geodesic segments $[x,y], [y,z],$ $[z,x]$ and  is denoted by $\Delta(x,y,z)$. For every geodesic triangle there exists a unique {\em tripod} $\overline \Delta$ with vertices $\bar{x},\bar{y},\bar{z}$ such that the lengths of $[\bar{x}, \bar{y}], [\bar{y}, \bar{z}], [\bar{z}, \bar{x}]$ equal the lengths of $[x,y], [y,z], [z,x]$ respectively. There exists a unique map $f_\Delta$ from $\Delta(x,y,z)$ to the tripod $\overline \Delta$ that identifies isometrically the corresponding edges, 
and there are exactly three points $c_x \in [y,z], c_y \in [x,z], c_z\in [x,y]$ such that $f_\Delta (c_x) = f_\Delta (c_y) = f_\Delta(c_z) = c$, where $c$ is the center of the tripod $\overline \Delta$. By definition of $f_\Delta$ it holds: 
$$d(x,c_z) = d(x,c_y), \qquad d(y,c_x) = d(y,c_z), \qquad d(z,c_x)=d(z,c_y).$$
The triangle $\Delta(x,y,z)$ is called {\em $\delta$-thin}
if for every $u,v \in \Delta(x,y,z)$ such that $f_\Delta(u)=f_\Delta(v)$ it holds $d(u,v)\leq \delta$;  in particular the mutual distances between $c_x,c_y$ and $c_z$ are at most $\delta$. 
It is well-known that  every geodesic triangle in a geodesic {\em $\delta$-hyperbolic} metric space (as defined above)  is $4\delta$-thin.


\noindent Moreover, the last condition is equivalent to the above definition of hyperbolicity, up to slightly increasing the hyperbolicity constant  $\delta$ in  (\ref{hyperbolicity}).

\noindent The following is a basic property of Gromov-hyperbolic metric spaces.
\begin{lemma}[Projection Lemma, cp. Lemma 3.2.7 of \cite{CDP90}]
	\label{projection}
	Let $X$ be a $\delta$-hyperbolic metric space and let $x,y,z \in X$. For every geodesic segment $[y,z]$ we have $(y,z)_x \geq d(x, [y,z]) - 4\delta.$
\end{lemma}
\noindent Let $X$ be a proper, $\delta$-hyperbolic metric space and $x$ be a point of $X$. \\
The {\em Gromov boundary} of $X$ is defined as the quotient 
$$\partial X = \lbrace (z_n)_{n \in \mathbb{N}} \subseteq X \hspace{1mm} | \hspace{1mm}   \lim_{n,m \to +\infty} (z_n,z_m)_{x} = + \infty \rbrace \hspace{1mm} /_\approx,$$
where $(z_n)_{n \in \mathbb{N}}$ is a sequence of points in $X$ and $\approx$ is the equivalence relation defined by $(z_n)_{n \in \mathbb{N}} \approx (z_n')_{n \in \mathbb{N}}$ if and only if $\lim_{n,m \to +\infty} (z_n,z_m')_{x} = + \infty$.  \linebreak
We will write $ z = [(z_n)] \in \partial X$ for short, and we say that $(z_n)$ {\em converges} to $z$. This definition  does not depend on the basepoint $x$. \\
There is a natural topology on $X\cup \partial X$ that extends the metric topology of $X$. 
The Gromov product can be extended to points $z,z'  \in \partial X$ by 
$$(z,z')_{x} = \sup_{(z_n) , (z_n') } \liminf_{n,m \to + \infty} (z_n, z_m')_{x}$$
where the supremum is taken among all sequences such that $(z_n) \in z$ and $(z_n')\in z'$.
For every $z,z',z'' \in \partial X$ it continues to hold
\begin{equation}
	\label{hyperbolicity-boundary}
	(z,z')_{x} \geq \min\lbrace (z,z'')_{x}, (z',z'')_{x} \rbrace - \delta.
\end{equation}
Moreover for all sequences $(z_n),(z_n')$ converging to  $z,z'$ respectively it holds
\begin{equation}
	\label{product-boundary-property}
	(z,z')_{x} -\delta \leq \liminf_{n,m \to + \infty} (z_n,z_m')_{x} \leq (z,z')_{x}.
\end{equation}
The Gromov product between a point $y\in X$ and a point $z\in \partial X$ is defined in a similar way and it  satisfies a condition analogue of  \eqref{product-boundary-property}.

\noindent Every geodesic ray $\xi$ defines a point  $\xi^+=[(\xi(n))_{n \in \mathbb{N}}]$  of the Gromov boundary $ \partial X$: we  say that $\xi$ {\em joins} $\xi(0) = y$ {\em to} $\xi^+ = z$, and we denote it by  $[y, z]$. Moreover for every $z\in \partial X$ and every $x\in X$ it is possible to find a geodesic ray $\xi$ such that $\xi(0)=x$ and $\xi^+ = z$. Any such geodesic ray is denoted as $\xi_{x,z} = [x,z]$ even if it is possibly not unique. 
Analogously, given different points $z = [(z_n)], z' = [(z'_n)] \in \partial X$ there always exists  a geodesic line $\gamma$ joining $z'$ to $z$, i.e. such that  $\gamma|_{[0, +\infty)}$ and $\gamma|_{(-\infty,0]}$ join   $\gamma(0)$ to $z,z'$ respectively (just  consider the limit $\gamma$ of the segments $[z_n,z'_n]$; notice that  all these segments intersect a ball   of fixed radius centered at $x_0$, since $(z_n,z'_m)_{x_0}$ is uniformly bounded above). We call $z$ and $z'$ the  {\em positive} and {\em negative endpoints} of $\gamma$, respectively,  denoted  $\gamma^\pm$.
The relation between Gromov product and geodesic ray is highlighted in the following well known lemma.
\begin{lemma}
	\label{product-rays}
	Let $X$ be a proper, $\delta$-hyperbolic metric space, $z,z'\in \partial X$ and $x\in X$. 
	\begin{itemize}
		\item[(i)] If $(z,z')_{x} \geq T$ then $d(\xi_{x,z}(T - \delta),\xi_{x,z'}(T - \delta)) \leq 4\delta$.
		\item[(ii)] If $d(\xi_{x,z}(T),\xi_{x,z'}(T)) < 2b$ then $(z,z')_{x} > T - b$, for all $b>0$.
	\end{itemize}
\end{lemma}
\begin{proof}
	Assume $(z,z')_{x} \geq T$ and suppose $d(\xi_{x,z}(T - \delta),\xi_{x,z'}(T - \delta)) > 4\delta$.
	Fix $S\geq T - \delta$ and consider the triangle $\Delta(x, \xi_{x,z}(S), \xi_{x,z'}(S))$. There exist $a\in [x,\xi_{x,z}(S)], b\in [x,\xi_{x,z'}(S)], c\in [\xi_{x,z}(S), \xi_{x,z'}(S)]$ such that $d(a,b)<\delta,\,\, d(b,c)<\delta,\,\, d(a,c)<\delta$ and
	$T_\delta := d(x,a)=d(x,b),$ $d(\xi_{x,z}(S),a) = d(\xi_{x,z}(S),c),\,\, d(\xi_{x,z'}(S),b) = d(\xi_{x,z'}(S),c).$
	Since this triangle is $4\delta$-thin we conclude that $T - \delta> T_\delta$. Moreover $d(\xi_{x,z}(S),\xi_{x,z'}(S)) = d(\xi_{x,z}(S),c) + d(c,\xi_{x,z'}(S)) = 2(S - T_\delta).$
	Hence
	\begin{equation*}
		\begin{aligned}
			(z,z')_{x} \leq \liminf_{S\to + \infty}\frac{1}{2}\big( 2S - d(\xi_{x,z}(S),\xi_{x,z'}(S)) \big) + \delta 
			= T_\delta + \delta < T
		\end{aligned}
	\end{equation*}
	where we have used \eqref{product-boundary-property}. This contradiction concludes (i).	\\
	Now we assume $d(\xi_{x,z}(T),\xi_{x,z'}(T)) < 2b$. Applying again \eqref{product-boundary-property} and using $d(\xi_{x,z}(S),\xi_{x,z'}(S)) < 2(S-T) + 2b$ for all $S\geq T$ we obtain
	$$(z,z')_{x} \geq \liminf_{S\to + \infty}\frac{1}{2}\big( 2S - d(\xi_{x,z}(S),\xi_{x,z'}(S)) \big) > T + b.$$
\end{proof}
\begin{obs}
	\label{rmk-product-finite}
	We remark that the computation above shows also that if $z\in \partial X$, $y\in X$ and $(y,z)_x \geq T$ then $d(x,y)> T-\delta$ and $d(\gamma(T-\delta), \xi_{x,z}(T-\delta)) \leq 4\delta$ for every geodesic segment $\gamma = [x,y]$.
\end{obs}
\noindent The following is a standard computation, see for instance \cite{BCGS}.
\begin{lemma}
	\label{parallel-geodesics}
	Let $X$ be a proper, $\delta$-hyperbolic metric space. Then every two geodesic rays $\xi, \xi'$ with same endpoints at infinity are at distance at most $8\delta$, i.e. there exist $t_1,t_2\geq 0$ such that $t_1+t_2=d(\xi(0),\xi'(0))$ and  $d(\xi(t + t_1),\xi'(t+t_2)) \leq 8\delta$ for all $t\in \mathbb{R}$.
\end{lemma}
A curve $\alpha\colon [a,b] \to X$ is a $(1,\nu)$-quasigeodesic, $\nu\geq 0$, if 
$$\vert s - t \vert - \nu \leq d(\alpha(s),\alpha(t)) \leq \vert s - t \vert + \nu$$
for all $s,t\in [a,b]$. A subset $Y$ of $X$ is said $\lambda$-quasiconvex if every point of every geodesic segment joining every two points $y,y'$ of $Y$ is at distance at most $\lambda$ from $Y$.
The {\em quasiconvex hull} of a subset $C$ of $\partial X$ is the union of all the geodesic lines joining two points of $C$ and it is denoted by QC-Hull$(C)$. The following lemma justifies this name.
\begin{lemma}
	\label{lemma-quasigeodesic}
	Let $X$ be a proper, $\delta$-hyperbolic metric space and let $C$ be a subset of $\partial X$. Then \textup{QC-Hull}$(C)$ is $36\delta$-quasiconvex. Moreover if $C$ is closed then \textup{QC-Hull}$(C)$ is closed.
\end{lemma}
\begin{proof}
	Let $x,y\in\text{QC-Hull}(C)$. By definition they belong to geodesics $\gamma_x,\gamma_y$ with both endpoints in $C$. We parametrize $\gamma_x$ and $\gamma_y$ in such a way that $d(\gamma_x(0),\gamma_y(0)) = d(\gamma_x,\gamma_y)$ and $x=\gamma_x(t_x)$, $y=\gamma_y(t_y)$ with $t_x,t_y \geq 0$. We take a geodesic $\gamma = [\gamma_x^+, \gamma_y^+] \subseteq \text{QC-Hull}(C)$. By Lemma \ref{parallel-geodesics} there are points $x',y'\in\gamma$ at distance at most $8\delta$ from $x$ and $y$ respectively. Therefore the path $\alpha = [x,x']\cup[x',y'] \cup [y',y]$ is a $(1,16\delta)$-quasigeodesic. By a standard computation in hyperbolic geometry (see for instance \cite{CavS20bis}, Proposition 3.5.(a)) we conclude that any point of $[x,y]$ is at distance at most $28\delta$ from a point of $\alpha$ and so at distance at most $36\delta$ from a point of $\gamma$. This concludes the proof of the first part since the points of $\gamma$ are in the quasiconvex hull of $C$.\
	Suppose now to have points $x_n\in \text{QC-Hull}(C)$ converging to $x_\infty \in X$. By definition $x_n \in \gamma_n$, where $\gamma_n$ is a geodesic line with endpoints $\gamma_n^\pm\in C$. The geodesics $\gamma_n$ converge uniformly on compact subsets to a geodesic $\gamma_\infty$ containing $x_\infty$, since $X$ is proper.
	The sequences $\gamma_n^\pm$ converge to the endpoints of $\gamma_\infty$ (cp. Lemma 1.6 of \cite{BL12}). Using the fact that $C$ is closed we conclude that $\gamma_\infty^\pm \in C$, i.e. $x_\infty \in \text{QC-Hull}(C)$.
\end{proof}

We need the following approximation result.
\begin{lemma}
	\label{approximation-ray-line}
	Let $X$ be a proper, $\delta$-hyperbolic metric space. Let $C\subseteq \partial X$ be a subset with at least two points and $x\in \textup{QC-Hull}(C)$. Then for every $z\in C$ there exists a geodesic line $\gamma$ with endpoints in $C$ such that $d(\xi_{x,z}(t), \gamma(t)) \leq 14\delta$ for every $t\geq 0$. In particular $d(\xi_{x,z}(t), \textup{QC-Hull}(C)) \leq 14\delta$.
\end{lemma}
\begin{proof}
	Since $x\in \textup{QC-Hull}(C))$ it exists a geodesic line $\eta$ joining two points $\eta^\pm$ of $C$ such that $x\in \eta$. Of course we have $(\eta^+,\eta^-)_x \leq \delta$, so by \eqref{hyperbolicity-boundary} we get
	$$\delta \geq (\eta^+,\eta^-)_x \geq \min\lbrace (\eta^+,z)_x, (\eta^-,z)_x\rbrace - \delta.$$
	Therefore one of the two values $(\eta^+,z)_x$, $(\eta^-,z)_x$ is $\leq 2\delta$. Let us suppose it is the first one. We consider a geodesic line $\gamma$ joining $\eta^+$ and $z$. By Lemma \ref{projection} we get
	$$d(x,\gamma([-S,S])) \leq (\gamma(-S), \gamma(S))_x + 4\delta$$
	for every $S\geq 0$. Taking $S\to +\infty$ the points $\gamma(-S)$ and $\gamma(S)$ converge respectively to $\eta^+$ and $z$. Therefore by \eqref{product-boundary-property} we get
	$d(x,\gamma) \leq 6\delta$. If we parametrize $\gamma$ so that $d(x,\gamma(0)) \leq 6\delta$ then $d(\xi_{x,z}(t), \gamma(t)) \leq 14\delta$ for every $t\geq 0$, by Lemma \ref{parallel-geodesics}.
\end{proof}

The {\em{Busemann function associated to $z\in \partial X$ with basepoint $x$}} is 
$$B_z(x,\cdot)\colon X \to \mathbb{R},\qquad y \mapsto \lim_{T\to +\infty} (d(\xi_{x,z}(T), y) - T).$$
It depends on the choice of the geodesic ray $\xi_{x,z}=[x,z]$ but two maps obtained taking two different geodesic rays are at bounded distance and the bound depends only on $\delta$. Every Busemann function is $1$-Lipschitz.

\subsection{Visual metrics}
\label{subsubsec-visual-metrics}
When $X$ is a proper, $\delta$-hyperbolic metric space it is known that the boundary $\partial X$ is metrizable. A metric $D_{x,a}$ on $\partial X$ is called a {\em visual metric} of center $x \in X$ and parameter $a\in\left(0,\frac{1}{2\delta\cdot\log_2e}\right)$ if there exists $V> 0$ such that for all $z,z' \in \partial X$ it holds
\begin{equation}
	\label{visual-metric}
	\frac{1}{V}e^{-a(z,z')_{x}}\leq D_{x,a}(z,z')\leq V e^{-a(z,z')_{x}}.
\end{equation}
A visual metric is said {\em standard} if for all $z,z'\in \partial X$ it holds
$$(3-2e^{a\delta})e^{-a(z,z')_{x}}\leq D_{x,a}(z,z')\leq e^{-a(z,z')_{x}}.$$
For all $a$ as before and $x\in X$ there exists always a standard visual metric of center $x$ and parameter $a$, see \cite{Pau96}.
The {\em generalized visual ball} of center $z \in \partial X$ and radius $\rho \geq 0$ is
$$B(z,\rho) = \bigg\lbrace z' \in \partial X \text{ s.t. } (z,z')_{x} > \log \frac{1}{\rho} \bigg\rbrace.$$
It is comparable to the metric balls of the visual metrics on $\partial X$.
\begin{lemma}
	\label{comparison-balls}
	Let $D_{x,a}$ be a visual metric of center $x$ and parameter $a$ on $\partial X$. Then for all $z\in \partial X$ and for all $\rho>0$ it holds
	$$B_{D_{x,a}}\left(z, \frac{1}{V}\rho^a\right)  \subseteq B(z,\rho)\subseteq B_{D_{x,a}}(z, V\rho^a ).$$
\end{lemma}
\begin{proof}
	If $z'\in B(z,\rho)$ then $(z,z')_{x} > \log \frac{1}{\rho}$, so $D_{x,a}(z,z')\leq Ve^{-a(z,z')_{x}} < V\rho^a.$ If $z'\in B_{D_{x,a}}(z,\frac{1}{V}\rho^a)$ then
	$\frac{1}{V}e^{-a(z,z')_{x}} \leq D_{x_0,a}(z,z') < \frac{1}{V}\rho^a$, i.e. $z'\in B(z,\rho)$.
\end{proof}
\noindent It is classical that generalized visual balls are related to shadows, whose definition is the following. The shadow of radius $r>0$ casted by a point $y\in X$ with center $x\in X$ is the set:
$$\text{Shad}_x(y,r) = \lbrace z\in \partial X \text{ s.t. } [x,z]\cap B(y,r) \neq \emptyset \text{ for all rays } [x,z]\rbrace.$$
\begin{lemma}
	\label{shadow-ball}
	Let $X$ be a proper, $\delta$-hyperbolic metric space. Let $z\in \partial X$, $x\in X$ and $T\geq 0$. Then 
	\begin{itemize}
		\item[(i)] $B(z,e^{-T}) \subseteq \textup{Shad}_{x}\left(\xi_{x,z}\left(T\right), 7\delta\right)$;
		\item[(ii)] $\textup{Shad}_{x}\left(\xi_{x,z}\left(T\right), r\right) \subseteq B(z, e^{-T + r})$ for all $r> 0$.
	\end{itemize}
\end{lemma}
\begin{proof}
	Let $z'\in B(z,e^{-T})$, i.e. $(z,z')_{x}> T$. By Lemma \ref{product-rays} we know that $d(\xi_{x,z}(T - \delta), \xi_{x,z'}(T - \delta)) \leq 4\delta.$
	So $d(\xi_{x,z'}(T), \xi_{x,z}(T)) \leq 6\delta < 7\delta$.
	This implies $z'\in \text{Shad}_{x}(\xi_{x,z}(T),7\delta)$, showing (i).\\
	Now we fix $z'\in \text{Shad}_{x}(\xi_{x,z}(T),r)$, which means that every geodesic ray $\xi_{x,z'}$ passes through $B(\xi_{x,z}(T), r)$, so $d(\xi_{x,z'}(T),\xi_{x,z}(T)) < 2r$. By Lemma \ref{product-rays} we conclude $(z,z')_{x} > T - r$, implying (ii).
\end{proof}


A compact metric space $Z$ is {\em $(A,s)$-Ahlfors regular} if there exists a probability measure $\mu$ on $Z$ such that
$$\frac{1}{A}\rho^s \leq \mu(B(z,\rho)) \leq A\rho^s$$
for all $z\in Z$ and all $0\leq \rho \leq \text{Diam}(Z)$, where Diam$(Z)$ is the diameter of $Z$.
In case $Z=\partial X$ we say that $Z$ is {\em visual $(A,s)$-Ahlfors regular} if there exists a probability measure $\mu$ on $\partial X$ such that
$$\frac{1}{A}\rho^s \leq \mu(B(z,\rho)) \leq A\rho^s$$
for all $z\in Z$ and all $0\leq \rho \leq 1$, where $B(z,\rho)$ is the generalized visual ball of center $z$ and radius $\rho$.
From Lemma \ref{comparison-balls} it follows immediately:
\begin{lemma}
	If $\partial X$ is $(A,s)$-Ahlfors regular with respect to a visual metric of center $x$ and parameter $a$, then it is visual $(AV^s,as)$-Ahlfors regular, where $V$ is the constant of \eqref{visual-metric}.
\end{lemma}

The packing$^*$ number at scale $\rho$ of a subset $C$ of the boundary of a proper Gromov-hyperbolic space $\partial X$ is the maximal number of disjoint generalized visual balls of radius $\rho$ with center in $C$. We denote it by $\text{Pack}^*(C, \rho)$. We write $\text{Cov}(C, \rho)$ to denote the minimal number of generalized visual balls of radius $\rho$ needed to cover $C$.
\begin{lemma}
	\label{packing*}
	For all $T\geq 0$ it holds $\textup{Pack}^*(C, e^{-T +\delta}) \leq \textup{Cov}(C, e^{-T})$ and $\textup{Cov}(C, e^{-T +\delta}) \leq \textup{Pack}^*(C, e^{-T}).$
\end{lemma}
\begin{proof}
	Let $z_1, \ldots, z_N$ be points of $C$ realizing $\textup{Cov}(C, e^{-T})$. Suppose there exist points $w_1,\ldots,w_M$ of $C$ such that $B(w_i, e^{-T +\delta})$ are disjoint, in particular $(w_i,w_j)_{x}\leq T - \delta$ for every $i\neq j$.
	If $M> N$ then two different points $w_i,w_j$ belong to the same ball $B(z_k, e^{-T})$, i.e.
	$(z_k,w_i)_{x}> T$ and $(z_k,w_j)_{x}> T.$
	By \eqref{hyperbolicity-boundary} we have
	$(w_i,w_j)_{x}> T - \delta$ which is a contradiction. This shows the first inequality. \\
	Now let $z_1,\ldots,z_N$ be a maximal collection of points of $C$ such that $B(z_i, e^{-T})$ are disjoint. Then for every $z\in C$ there exists $i$ such that $B(z,e^{-T}) \cap B(z_i,e^{-T}) \neq \emptyset$. Therefore there exists $w\in \partial X$ such that $(z_i,w)_{x}> T$ and $(z,w)_{x}> T$.
	As before we get $(z_i,z)_{x}> T- \delta,$ proving the second inequality.
\end{proof}

\subsection{Groups of isometries, limit set and critical exponent}
Let $X$ be a proper, $\delta$-hyperbolic metric space.
Every isometry of $X$ acts naturally on $\partial X$ and the resulting map on $X\cup \partial X$ is a homeomorphism.
The {\em limit set} $\Lambda(\Gamma)$ of a discrete group of isometries $\Gamma$ is the set of accumulation points of the orbit $\Gamma x$ on $\partial X$, where $x$ is any point of $X$. It is the smallest $\Gamma$-invariant closed set of the Gromov boundary, indeed:
\begin{prop}[\cite{Coo93}, Theorem 5.1]
	\label{prop-minimality-limit-set}
	Let $\Gamma$ be a discrete group of isometries of a proper, Gromov-hyperbolic metric space. Then $\Lambda(\Gamma)$ is the smallest closed $\Gamma$-invariant subset of $\partial X$, i.e. every $\Gamma$-invariant, closed subset $C$ of $\partial X$ contains $\Lambda(\Gamma)$.
\end{prop}

\noindent The group $\Gamma$ is called {\em elementary} if $\# \Lambda(\Gamma) \leq 2$. The limit superior in the definition of the critical exponent of $\Gamma$ is a true limit.
\begin{lemma}[\cite{Cav21bis}, Theorem B]
	\label{entropy-critical}
	Let $X$ be a proper, $\delta$-hyperbolic metric space and let $\Gamma$ be a discrete group of isometries of $X$. Then $$h_\Gamma = \lim_{T\to + \infty} \frac{1}{T} \log \# \Gamma x \cap \overline{B}(x,T).$$
\end{lemma}
\noindent The {\em critical exponent} of $\Gamma$ can be seen also as  
$$h_\Gamma = \inf\bigg\lbrace s \geq 0 \text{ s.t. } \sum_{g \in \Gamma}e^{-sd(x,g x)} < +\infty\bigg\rbrace.$$
We remark that for every $s \geq 0$ the series $\sum_{g \in \Gamma}e^{-sd(x,g x)}$, which is called the Poincaré series of $\Gamma$, is $\Gamma$-invariant. In other words $\sum_{g \in \Gamma}e^{-sd(x,g x)} = \sum_{g \in \Gamma}e^{-sd(x',g x')}$ for all $x' \in \Gamma x$. 
\noindent There is a canonical way to construct a measure on $\partial X$ starting from the Poincaré series. For every $s>h_\Gamma$ the measure
$$\mu_s = \frac{1}{\sum_{g \in \Gamma}e^{-sd(x,g x)}}\cdot \sum_{g \in \Gamma}e^{-sd(x,g x)}\Delta_{g x},$$
where $\Delta_{g x}$ is the Dirac measure at $g x$, is a probability measure on the compact space $X\cup \partial X$. Then there exists a sequence $s_i$ converging to $h_\Gamma$ such that $\mu_{s_i}$ converges $*$-weakly to a probability measure on $X\cup \partial X$. Any of these limits is called a Patterson-Sullivan measure and it is denoted by $\mu_{\text{PS}}$. 
\begin{prop}[Theorem 5.4 of \cite{Coo93}.]
	\label{patterson-sullivan}
	Let $X$ be a proper, $\delta$-hyperbolic metric space and let $\Gamma$ be a discrete group of isometries of $X$ with $h_\Gamma<+\infty$. Then every Patterson-Sullivan measure is supported on $\Lambda(\Gamma)$. Moreover it is a $\Gamma$-quasiconformal density of dimension $h_\Gamma$, i.e. it satisfies
	$$\frac{1}{Q}\cdot e^{h_\Gamma (B_z(x,x) - B_z(x,gx))} \leq \frac{d(g_* \mu_{\textup{PS}})}{d\mu_{\textup{PS}}}(z) \leq Q\cdot e^{h_\Gamma (B_z(x,x) - B_z(x,gx))}$$
	for every $g \in \Gamma$ and every $z\in \Lambda(\Gamma)$, where $Q$ is a constant depending only on $\delta$ and an upper bound on $h_\Gamma$.
\end{prop}
\noindent The quantification of $Q$ is not explicitated in the original paper, but it follows from the proof therein.\vspace{2mm}

\noindent The set $\Lambda(\Gamma)$ is $\Gamma$-invariant so it is its quasiconvex hull. We recall that a discrete group of isometries $\Gamma$ is {\em quasiconvex-cocompact} if and only if its action on QC-Hull$(\Lambda(\Gamma))$ is cocompact, i.e. if there exists $D\geq 0$ such that for all $x,y\in \text{QC-Hull}(\Lambda(\Gamma))$ it holds $d(gx,y)\leq D$ for some $g\in \Gamma$. The smallest $D$ satisfying this property is called the {\em codiameter} of $\Gamma$.\vspace{2mm}

\noindent\emph{Given two real numbers $\delta \geq 0$ and $D>0$ we recall that $\mathcal{M}(\delta,D)$ is the class of triples $(X,x,\Gamma)$, where $X$ is a proper, geodesic, $\delta$-hyperbolic metric space, $\Gamma$ is a discrete, non-elementary, torsion-free, quasiconvex-cocompact group of isometries with codiameter $\leq D$ and $x\in \textup{QC-Hull}(\Lambda(\Gamma))$. For an element $(X,x,\Gamma)$ of $\mathcal{M}(\delta,D)$ we will use $Y$ to denote $\textup{QC-Hull}(\Lambda(\Gamma))$.}

\section{$\mathcal{M}(\delta,D)$ is closed under equivariant GH-limits}
The purpose of this section is to prove statement (i) of Theorem \ref{theorem-main}. We need to understand better the properties of the spaces belonging to $\mathcal{M}(\delta,D)$.

\subsection{Entropy and systolic estimates on $\mathcal{M}(\delta,D)$}

The following are straightforward adaptations of results of \cite{BCGS} and \cite{BCGS2}. 

\begin{lemma}
	\label{lemma-generating-set}
	If $(X,x,\Gamma) \in \mathcal{M}(\delta,D)$ then
	$\Sigma_{2D+72\delta}(x)$ generates $\Gamma$.
\end{lemma}
\begin{proof}
	The proof is classical for geodesic metric spaces. In this setting we need to use the fact that $\text{QC-Hull}(\Lambda(\Gamma))$ is $36\delta$-quasiconvex. By discreteness we can fix a small $\varepsilon >0$ such that $d(x,gx) < 2D+72\delta + \varepsilon$ implies $d(x,gx)\leq 2D+72\delta$ for all $g\in\Gamma$. We take any $g\in \Gamma$ and we take consecutive points $x_i$, $i=0,\ldots,N$, on a geodesic segment $[x,gx]$ such that $x_0=x$, $x_N=gx$ and $d(x_i,x_{i+1}) <\varepsilon$. By Lemma \ref{lemma-quasigeodesic} each $x_i$ is at distance at most $36\delta$ from a point $y_i \in \text{QC-Hull}(\Lambda(\Gamma))$, hence there exists some $h_i\in\Gamma$ such that $d(h_ix,x_i)\leq 36\delta + D$. We can choose $h_N = g$ and $h_0=\text{id}$. We define the elements $g_i =  h_{i-1}^{-1} h_i$ for $i = 1,\ldots,N$. Clearly $g_1\cdots g_{N-1}g_N = g$. Moreover $d(g_i x, x) < 72\delta + 2D + \varepsilon$ for every $i$, so $g_i \in \Sigma_{2D+72\delta}(x)$.
\end{proof}

\begin{prop}[\cite{BCGS}, Proposition 5.10]
	\label{prop-lower-bound-entropy}
	If $(X,x,\Gamma) \in \mathcal{M}(\delta,D)$ then $h_\Gamma \geq \frac{\log 2}{99\delta + 10D}$.
\end{prop}
\begin{proof}
	Using the same proof of Lemma 5.14 of \cite{BCGS} we conclude that there exists a hyperbolic isometry $a\in\Gamma$ such that $\ell(a)\leq 8D + 10\delta$. The remaining part of the proof can be done exactly in the same way of \cite{BCGS}, choosing $y\in \text{Min}(a) \subseteq \text{QC-Hull}(\Lambda(\Gamma))$ and using Lemma \ref{lemma-generating-set}.
\end{proof}
\begin{prop}[\cite{BCGS2}, Theorem 3.4]
	\label{prop-bound-systole}
	For every $H \geq 0$ there exists $s=s(\delta,D,H) > 0$ such that if $(X,x,\Gamma) \in \mathcal{M}(\delta,D)$ and if $h_\Gamma \leq H$ then \textup{sys}$(\Gamma,X) \geq s.$
\end{prop}
\begin{proof}
	The proof is the same of Theorem 3.4 of \cite{BCGS2}. The only non-trivial part is the Bishop-Gromov estimate stated in Theorem 3.1 of \cite{BCGS2} and proved in \cite{BCGS}, Theorem 5.1. It is made in the cocompact case but it extends word by word to the quasiconvex-cocompact setting.
\end{proof}

\subsection{Covering entropy}
Let $Y$ be any subset of a metric space $X$:\\
-- a subset $S$ of $Y$ is called  {\em $r$-dense}   if   for all $y \in Y$  there exists $z\in S$ such that $d(y,z)\leq r$; \\
-- a subset $S$ of $Y$ is called  {\em $r$-separated} if  $d(y,z)> r$  for all $y,z \in S$.\\
The packing number of $Y$ at scale $r$ is the maximal cardinality of a $2r$-separated subset of $Y$ and it is denoted by $\text{Pack}(Y,r)$. The covering number of $Y$ is the minimal cardinality of a $r$-dense subset of $Y$ and it is denoted by $\text{Cov}(Y,r)$. These two quantities are classically related by:
\begin{equation}
	\label{packing-covering}
	\text{Pack}(Y,2r) \leq \text{Cov}(Y,2r) \leq \text{Pack}(Y,r).
\end{equation}
$Y$ is said uniformly packed at scales $0<r\leq R$ if 
$$\text{Pack}_Y(R,r) := \sup_{x\in Y}\text{Pack}(\overline{B}(x,R) \cap Y,r) <+\infty$$ 
and uniformly covered at scales $0<r\leq R$ if
$$\text{Cov}_Y(R,r):=\sup_{x\in Y}\text{Cov}(\overline{B}(x,R) \cap Y,r) <+\infty.$$

\begin{lemma}
	\label{lemma-uniformly-covered}
	Let $(X,x,\Gamma)\in \mathcal{M}(\delta,D)$. Then $Y=\textup{QC-Hull}(\Lambda(\Gamma))$ is uniformly packed and uniformly covered at any scales.
\end{lemma}
\begin{proof}
	We prove only that $Y$ is uniformly covered since the other case is similar. We fix $0<r\leq R$. The map $y\mapsto \text{Cov}(\overline{B}(y,R) \cap Y,r)$ defined on $Y$ is clearly $\Gamma$-invariant. If the thesis is false we could find a sequence of points $x_n\in Y$ such that $\text{Cov}(\overline{B}(x_n,R) \cap Y,r) \geq n$. By $\Gamma$-invariance and the compactness of the quotient we can suppose that $x_n$ converges to some point $x_\infty$ that belongs to $Y$ by Lemma \ref{lemma-quasigeodesic}. Clearly we would have $\text{Cov}(\overline{B}(x_\infty,R+1) \cap Y,r) = \infty$ which is impossible since $\overline{B}(x_\infty \cap Y,R+1)$ is compact.
\end{proof}

We recall the notion of covering entropy. It has been studied by the author in a less general context in \cite{Cav21}.

\begin{defin}
	Let $X$ be a proper metric space and $x\in X$. The \emph{upper covering entropy} of $X$ at scale $r>0$ is the quantity
	$$\overline{h}_\textup{Cov}(X,r) = \limsup_{T\to + \infty}\frac{\log \textup{Cov}(\overline{B}(x,T),r)}{T},$$
	while the \emph{lower covering entropy} of $X$ at scale $r$ is
	$$\underline{h}_\textup{Cov}(X,r) = \liminf_{T\to + \infty}\frac{\log \textup{Cov}(\overline{B}(x,T),r)}{T}.$$
	They do not depend on the point $x\in X$ by a standard argument.
\end{defin}

\begin{lemma}
	\label{lemma-bound-entropy-packing}
	Let $(X,x,\Gamma)\in \mathcal{M}(\delta,D)$. If there exist $r,P>0$ such that \textup{Pack}$_{Y}(72\delta + 3r, r) \leq P$ then \textup{Pack}$_{Y}(T, r) \leq P\cdot(1+P)^{\frac{T}{r} - 1}$ for every $T\geq 0$. In particular $\overline{h}_\textup{Cov}(Y,2r) \leq \frac{\log (1+P)}{r}$.
\end{lemma}
\begin{proof}
	The proof is the same of Lemma 4.7 of \cite{CavS20}, except from the fact that $Y$ is not geodesic but only $36\delta$-quasigeodesic by Lemma \ref{lemma-quasigeodesic}. We proceed by induction on $k$, where $k$ is the smallest integer such that $T\leq 72\delta + 3r + kr$. For $k=0$ the result  is obvious by our assumption. The inductive step goes as follows: by induction we can find a maximal $2r$-separated subset $\lbrace y_1,\ldots,y_N\rbrace$ of $\overline{B}(x,T-r)\cap Y$ with $N\leq P(1+P)^{\frac{T-r}{r}-1}$. The key step is to show that $\bigcup_{i=1}^N \overline{B}(y_i,72\delta+3r) \supseteq A(x,T-r,T)\cap Y$, where $A(x,T-r,T)$ is the closed annulus centered at $x$ of radii $T-r$ and $T$. Indeed for every point $y\in A(x,T-r,T)$ we consider the point $y'$ along a geodesic segment $[x,y]$ at distance $T - r - 36\delta$ from $x$. By quasiconvexity there is a point $z\in Y$ at distance $\leq 36\delta$ from $y'$. In particular $z\in \overline{B}(x,T-r)$, so $d(z,y_i)\leq 2r$ for some $i=1,\ldots,N$. We conclude that $d(y,y_i)\leq 72\delta + 3r$. The rest of the proof can be done exactly as in Lemma 4.7 of \cite{CavS20}, while the estimate on the upper covering entropy follows trivially using \eqref{packing-covering}.
\end{proof}

\begin{prop}
	\label{prop-critical-exponent-covering-entropy}
	Let $(X,x,\Gamma)\in \mathcal{M}(\delta,D)$. Then $\overline{h}_\textup{Cov}(Y,r)$ and $\underline{h}_\textup{Cov}(Y,r)$ do not depend on $r$
	and the same quantities can be defined replacing the covering function with the packing function.
	Moreover they coincide and
	$$h_\textup{Cov}(Y):=\lim_{T\to + \infty}\frac{\log \textup{Cov}(\overline{B}(x,T)\cap Y,r)}{T} = h_\Gamma < +\infty.$$
\end{prop}
\begin{proof}
	Let us fix $0<r\leq r'$. We have:
	$$\textup{Cov}(\overline{B}(x,T)\cap Y,r') \leq \textup{Cov}(\overline{B}(x,T)\cap Y,r)$$
	and 
	$$\textup{Cov}(\overline{B}(x,T)\cap Y,r) \leq \textup{Cov}(\overline{B}(x,T)\cap Y,r') \cdot \text{Cov}_Y(r',r).$$
	The quantity $\text{Cov}_Y(r',r)$ is finite by Lemma \ref{lemma-uniformly-covered}. These inequalities easily imply that $\overline{h}_\textup{Cov}(Y,r) = \overline{h}_\textup{Cov}(Y,r')$ and $\underline{h}_\textup{Cov}(Y,r) = \underline{h}_\textup{Cov}(Y,r')$.
	Moreover by \eqref{packing-covering} these quantities can be defined replacing the covering function with the packing function. Furthermore an application of Lemma \ref{lemma-uniformly-covered} and Lemma \ref{lemma-bound-entropy-packing} shows that the upper covering entropy of $Y$ is finite.\\
	Let $2s=\text{sys}(\Gamma,X) >0$. We have
	$$\text{Cov}(\overline{B}(x,T)\cap Y,D) \leq \#\Gamma x \cap \overline{B}(x,T+D)$$
	and
	$$\text{Pack}(\overline{B}(x,T)\cap Y,s) \geq \#\Gamma x \cap \overline{B}(x,T).$$
	Observe that the sequence $\#\Gamma x \cap \overline{B}(x,T)$ converges to $h_\Gamma$ when $T$ goes to $+\infty$ by Lemma \ref{entropy-critical}.
	Therefore $\overline{h}_\textup{Cov}(Y,D) \leq h_\Gamma$ and $\underline{h}_\textup{Cov}(Y,s) \geq h_\Gamma$, implying the last part of the thesis.
\end{proof}

\subsection{Convergence of spaces in $\mathcal{M}(\delta,D)$}
The following situation will be called the \emph{standard setting of convergence}: we have a sequence $(X_n,x_n,\Gamma_n) \in \mathcal{M}(\delta,D)$ such that $(X_n,x_n,\Gamma_n) \underset{\textup{eq-pGH}}{\longrightarrow} (X_\infty, x_\infty,\Gamma_\infty)$.
Observe that $X_\infty$ is a proper metric space by definition.

\begin{lemma}
	\label{lemma-uniform-packing-convergence}
	 In the standard setting of convergence $\sup_{n\in\mathbb{N}}\textup{Pack}_{Y_n}(R,r)< +\infty$ for every $0<r\leq R$.
\end{lemma}
\begin{proof}
	By Gromov precompactness Theorem \cite{Gr81} we know that $$\sup_{n\in\mathbb{N}}\textup{Pack}(\overline{B}(x_n,R+D)\cap Y_n,r) =: P < +\infty.$$
	For every $n\in\mathbb{N}$ and every point $y_n \in Y_n$ there is some $g\in \Gamma_n$ such that $d(y_n,g_nx_n)\leq D$. Therefore 
	\begin{equation*}
		\begin{aligned}
			\textup{Pack}(\overline{B}(y_n,R) \cap Y_n,r) &\leq \textup{Pack}(\overline{B}(gx_n,R+D)\cap Y_n,r) \\
			&= \textup{Pack}(\overline{B}(x_n,R+D)\cap Y_n,r) \leq P
		\end{aligned}
	\end{equation*}
	by the $\Gamma_n$-invariance of $Y_n$.
\end{proof}

\begin{cor}
	\label{cor-entropy-bound-convergence}
	In the standard setting of convergence $\sup_{n\in\mathbb{N}}h_{\Gamma_n} < +\infty$.
\end{cor}
\begin{proof}
	By Lemma \ref{lemma-uniform-packing-convergence} we have $\sup_{n\in\mathbb{N}}\text{Pack}_{Y_n}(72\delta + 3,1) =: P < +\infty$. Therefore by Proposition \ref{prop-critical-exponent-covering-entropy} and Lemma \ref{lemma-bound-entropy-packing} we have
	$$h_{\Gamma_n} = h_\textup{Cov}(Y_n) \leq \log(1+P).$$
\end{proof}

\begin{cor}
	\label{cor-discrete-limit}
	In the standard setting of convergence $\Gamma_\infty$ is discrete and torsion-free.
\end{cor}
\begin{proof}
	By Corollary \ref{cor-entropy-bound-convergence} and Proposition \ref{prop-bound-systole} there is some $s>0$ such that $\text{sys}(\Gamma_n,X_n) \geq s$ for every $n\in \mathbb{N}$.
	Let $\omega$ be a non-principal ultrafilter. By Proposition \ref{prop-GH-ultralimit} it is enough to show that $\Gamma_\omega$ is discrete and torsion-free.
	Let $g_\omega = \omega$-$\lim g_n$ be a non-trivial element of $\Gamma_\omega$ and $y_\omega = \omega$-$\lim y_n$ be a point of $X_\omega$. We know that $d(y_n, g_n y_n )\geq s$ for $\omega$-a.e.$(n)$. This implies $d(y_\omega, g_\omega y_\omega )\geq s$. Since this is true for every $g_\omega \in \Gamma_\omega$ and every $y_\omega \in X_\omega$ we conclude that sys$(\Gamma_\omega, X_\omega) \geq s$. Since $X_\omega$ is proper we conclude that $\Gamma_\omega$ is discrete. Take now an elliptic element $g_\omega=\omega$-$\lim g_n$. It is classical that $g_\omega$ must have finite order since $\Gamma_\omega$ is discrete (see for instance Remark 8.16 of \cite{BCGS}), i.e. $g_\omega^k =\text{id}$ for some $k\in\mathbb{Z}\setminus \lbrace 0 \rbrace$.
	This means that $\omega$-$\lim d(g_n^k x_n, x_n) = 0$, so $g_n^k = \text{id}$ for $\omega$-a.e.$(n)$. This implies $g_n = \text{id}$ for $\omega$-a.e.$(n)$ and therefore $g_\omega = \text{id}$. In other words $\Gamma_\omega$ is torsion-free.
\end{proof}

The next step is to show the stability of the boundary under convergence.
\begin{prop}
	\label{boundary-convergence}
	Let $(X_n,x_n)$ be a sequence of proper, $\delta$-hyperbolic metric spaces and let $D_{x_n,a}$ be a standard visual metric of center $x_n$ and parameter $a$ on $\partial X_n$. Let $\omega$ be a non-principal ultrafilter and let $(X_\omega, x_\omega)$ be the ultralimit of the sequence $(X_n,x_n)$. Then there exists a natural map $\Psi\colon \omega$-$\lim (\partial X_n, D_{x_n,a}) \to \partial X_\omega$ which is a homeomorphism onto the image.
\end{prop}
\begin{obs}
	We point out that:
	\begin{itemize}
		\item[(1)] the spaces $\partial X_n$ are compact with diameter at most $1$ then the ultralimit $\omega$-$\lim \partial X_n$ does not depend on the basepoints.
		\item[(2)] In general the map $\Psi$ is not surjective: let $X_n$ be the closed ball $\overline{B}(o,n)$ inside the hyperbolic plane $\mathbb{H}^2$, where $o$ is a fixed basepoint. Each $X_n$ is proper and $\delta$-hyperbolic for the same $\delta$, but $\partial X_n = \emptyset$. Therefore $\omega$-$\lim \partial X_n = \emptyset$. On the other hand $X_\omega = \mathbb{H}^2$ and $\partial X_\omega \neq \emptyset$.
		\item[(3)] It is possible to prove, but we will not do it because it is not necessary to our scope, that if for each point $y_n$ of $X_n$ there is a geodesic ray $[x_n,z_n]$ passing at distance $\leq \delta$ from $y_n$ then the map $\Psi$ is surjective. Moreover when $\Psi$ is surjective then the metric induced on $\partial X_\omega$ by $\Psi$ is a visual metric of center $x_\omega$ and parameter $a$.
	\end{itemize}
\end{obs} 
\begin{proof}
	A point of $\omega$-$\lim \partial X_n$ is a class of a sequence of points $(z_n) \in \partial X_n$ and for each point $z_n$ there exists a geodesic ray $\xi_{x_n,z_n}$.
	The sequence of geodesic rays $(\xi_{x_n,z_n})$ defines an ultralimit geodesic ray $\xi$ of $X_\omega$ with $\xi(0)=x_\omega$ (cp. \cite{CavS20}, Lemma A.7) which provides a point of $\partial X_\omega$. We denote this point by $z_\omega$ and $\xi$ by $\xi_{x_\omega,z_\omega}$.
	We define the map $\Psi \colon \omega$-$\lim \partial X_n \to \partial X_\omega$ as $\Psi((z_n)) = \xi_{x_\omega,z_\omega}^+ = z_\omega$. \\
	{\textbf{Good definition.}} We need to show that $\Psi$ is well defined, i.e. it does not depend on the choice of the geodesic ray $\xi_{x_n,z_n}$ and on the choice of the sequence $(z_n)$. Let $(z_n')$ be a sequence of points equivalent to $(z_n)$, i.e. $\omega\text{-}\lim D_{x_n,a}(z_n,z_n') = 0$. Choose geodesic rays $\xi_{x_n,z_n}$, $\xi_{x_n,z_n'}$.
	For every $n$ the metric $D_{x_n,a}$ is a standard visual metric, then for every fixed $\varepsilon > 0$ it holds $(z_n,z_n')_{x_n} > \log\frac{1}{\varepsilon} =: T_\varepsilon$ for $\omega$-a.e.$(n)$. Thus $d(\xi_{x_n,z_n}(T_\varepsilon - \delta), \xi_{x_n,z_n'}(T_\varepsilon - \delta)) \leq 4\delta$ by Lemma \ref{product-rays}, $\omega$-a.s. We conclude that $d(\xi_{x_\omega,z_\omega}(T_\varepsilon - \delta), \xi_{x_\omega,z_\omega'}(T_\varepsilon - \delta)) < 6\delta$ $\omega$-a.s. Therefore, again by Lemma \ref{product-rays}, $(z_\omega,z_\omega')_{x_\omega} > T_\varepsilon - 4\delta$. Thus $(z_\omega,z_\omega')_{x_\omega} = +\infty$ by the arbirariness of $\varepsilon$, i.e. $z_\omega = z_\omega'$.\\
	\textbf{Injectivity.} The next step is to show that $\Psi$ is injective. If two sequence of points $(z_n), (z_n')$ have the same image under $\Psi$ then $(\xi_{x_\omega,z_\omega}^+, \xi_{x_\omega,z_\omega'}^+)_{x_\omega} = +\infty$. So $(\xi_{x_\omega,z_\omega}^+, \xi_{x_\omega,z_\omega'}^+)_{x_\omega} \geq T$ for every fixed $T\geq 0$. Hence $d(\xi_{x_\omega,z_\omega}(T - \delta), \xi_{x_\omega,z_\omega'}(T - \delta)) \leq 4\delta$, by Lemma \ref{product-rays}. Then $d(\xi_{x_n,z_n}(T - \delta), \xi_{x_n,z_n'}(T - \delta)) < 6\delta$ $\omega$-a.s., i.e. $(z_n,z_n')_{x_n} > T- 4\delta$ $\omega$-a.s., again by Lemma \ref{product-rays}.
	Therefore $D_{x_n,a}(z_n,z_n')\leq e^{-a(T - 4\delta)}$. Since this is true $\omega$-a.s. we get $\omega$-$\lim D_{x_n,a}(z_n,z_n') \leq e^{-a(T - 4\delta)}$ for $\omega$-a.e.$(n)$. By the arbitrariness of $T$ we deduce that $\omega$-$\lim D_{x_n,a}(z_n,z_n') = 0$, i.e. $(z_n) = (z_n')$ as elements of $\omega$-$\lim \partial X_n$. \\
	\textbf{Homeomorphism.} Let us show $\Psi$ is continuous. Both $\omega$-$\lim \partial X_n$ and $\partial X_\omega$ are metrizable, then it is enough to check the continuity on sequences of points. We take a sequence $(z_n^k)_{k\in\mathbb{N}}$ converging to $(z_n^\infty)$ in $\omega$-$\lim \partial X_n$. By definition for every $\varepsilon > 0$ there exists $k_\varepsilon \geq 0$ such that if $k\geq k_\varepsilon$ then $\omega$-$\lim D_{x_n,a}(z_n^k, z_n^\infty) <\varepsilon$. 
	Therefore for every fixed $k\geq k_\varepsilon$ it holds $(z_n^k,z_n^\infty)_{x_n} \geq \log\frac{1}{\varepsilon} =: T_\varepsilon$ for $\omega$-a.e.$(n)$. As usual we conclude that $d(\xi_{x_nz_n^k}(T_\varepsilon - \delta),\xi_{x_nz_n^\infty}(T_\varepsilon - \delta)) \leq 4\delta$  $\omega$-a.s. Thus $d(\xi_{x_\omega z_\omega^k}(T_\varepsilon - \delta),\xi_{x_\omega,z_\omega^\infty}(T_\varepsilon - \delta)) < 6\delta$ for every fixed $k\geq k_\varepsilon$. Again this implies $(z_\omega^k,z_\omega^\infty)_{x_\omega} > T_\varepsilon - 4\delta$ for all $k\geq k_\varepsilon$. By the arbitrariness of $\varepsilon$ we get that $z_\omega^k$ converges to $z_\omega^\infty$ when $k$ goes to $+\infty$.\\
	The continuity of the inverse map defined on the image of $\Psi$ can be proved in a similar way.
\end{proof}

\begin{proof}[Proof of Theorem \ref{theorem-main}.(i).]
	In order to simplify the notations we fix a non-principal ultrafilter $\omega$. We know that $(X_\omega, x_\omega, \Gamma_\omega)$ is equivariantly isometric to $(X_\infty, x_\infty, \Gamma_\infty)$ by Proposition \ref{prop-GH-ultralimit} and that $X_\omega$ is a proper metric space, so we can prove all the properties for this triple. It is classical that the ultralimit of geodesic, $\delta$-hyperbolic metric spaces is a geodesic and $\delta$-hyperbolic metric space, see for instance \cite{DK18}. Moreover by Corollary \ref{cor-discrete-limit} the group $\Gamma_\omega$ is discrete and torsion-free.\\
	Let $\Psi$ be the homeomorphism onto the image given by Proposition \ref{boundary-convergence}. We claim that $\Lambda(\Gamma_\omega) = \Psi(\omega$-$\lim \Lambda(\Gamma_n))$.
	We fix a sequence $z_n \in \Lambda(\Gamma_n)$ and we observe that by Lemma \ref{approximation-ray-line} and the cocompactness of the action of $\Gamma_n$ on $\text{QC-Hull}(\Lambda(\Gamma_n))$ we can find a sequence $(g_n^k)_{k\in \mathbb{N}}\subseteq \Gamma_n$ such that, denoted by $\xi_{x_n,z_n}$ one geodesic ray $[x_n,z_n]$, it holds:
	\begin{itemize}
		\item[(a)] $g_n^k x_n$ converges to $z_n$ when $k$ tends to $+\infty$;
		\item[(b)] $g_n^0 = \text{id}$;
		\item[(c)] $d(g_n^k x_n, g_n^{k+1}x_n)\leq 28\delta + 2D$;
		\item[(d)] $d(g_n^k x_n, \xi_{x_n,z_n}(k)) \leq 14\delta + D$.
	\end{itemize}
	For every $k \in \mathbb{N}$ the sequence $g_n^k$ is admissible by (b) and (c), so it defines a limit isometry $g_\omega^k \in \Gamma_\omega$. Moreover, if $\xi_{x_\omega,z_\omega}$ is the ultralimit of the sequence of geodesic rays $\xi_{x_n,z_n}$, we have $d(g_\omega^k x_\omega, \xi_{x_\omega,z_\omega}(k)) \leq 14\delta + D$ for every $k\in \mathbb{N}$. Observe that $\xi_{x_\omega,z_\omega}^+ = \Psi((z_n))$ by definition of $\Psi$. As a consequence the sequence $g_\omega^k x_\omega$ converges to $\Psi((z_n))$, i.e. $\Psi(z_n)\in \Lambda(\Gamma_\omega)$. This shows that $\Psi(\omega\text{-}\lim \Lambda(\Gamma_n)) \subseteq \Lambda(\Gamma_\omega)$. Clearly $\Gamma_\omega$ acts on $\omega\text{-}\lim \Lambda(\Gamma_n)$ by $(g_n)(z_n) = (g_n z_n)$ and this action commutes with $\Psi$. The set $\omega\text{-}\lim \Lambda(\Gamma_n)$ is $\Gamma_\omega$-invariant and closed. The $\Gamma_\omega$-invariance is trivial, so let us check the closure. If $(z_n^k)_{k\in \mathbb{N}} \in \omega\text{-}\lim \Lambda(\Gamma_n)$ is a sequence converging to $(z_n^\infty)$ and $z_n^\infty \notin \omega\text{-}\lim \Lambda(\Gamma_n)$ then there exists $\varepsilon > 0$ such that $D_{x_n,a}(z_n^\infty, \Lambda(\Gamma_n)) \geq \varepsilon$ $\omega$-a.s. This is a contradiction. Therefore also the set $\Psi(\omega\text{-}\lim \Lambda(\Gamma_n))$ is closed and $\Gamma_\omega$-invariant. By Proposition \ref{prop-minimality-limit-set} we conclude that $\Lambda(\Gamma_\omega) = \Psi(\omega$-$\lim \Lambda(\Gamma_n))$.
	This also implies that $\omega$-$\lim \text{QC-Hull}(\Lambda(\Gamma_n)) = \text{QC-Hull}(\Lambda(\Gamma_\omega))$ and so $x_\omega \in \text{QC-Hull}(\Lambda(\Gamma_\omega))$.
	For every two points $y_\omega, y_\omega' \in \text{QC-Hull}(\Lambda(\Gamma_\omega))$ there exist sequences of points $y_n, y_n' \in \text{QC-Hull}(\Lambda(\Gamma_n))$ such that $y_\omega = \omega$-$\lim y_n$ and $y_\omega' = \omega$-$\lim y_n'$. So there exists $g_n \in \Gamma_n$ such that $d(g_ny_n,y_n')\leq D$. The sequence $g_n$ is clearly admissible so it defines an element $g_\omega = \omega$-$\lim g_n$ of $\Gamma_\omega$ and $d(g_\omega y_\omega, y_\omega')\leq D$, implying that the action of $\Gamma_\omega$ on $\text{QC-Hull}(\Lambda(\Gamma_\omega))$ is cocompact with codiameter $\leq D$.	
	It remains only to show that $\Gamma_\omega$ is non-elementary. If $\Gamma_\omega$ is elementary then $\text{QC-Hull}(\Lambda(\Gamma_\omega)) = \mathbb{R}$ and $\Gamma_\omega$ acts on $\mathbb{R}$ as $\mathbb{Z}_\tau$, the group generated by the translation of length $\tau$, for some $\tau > 0$. Denote by $g_\omega = \omega$-$\lim g_n$ the element corresponding to this translation. For every $k\in\mathbb{N}$ we notice that 
	$$A_\omega(k)=\left\lbrace h_\omega \in \Gamma_\omega \text{ s.t. } d(x_\omega, h_\omega x_\omega) < (k+1)\cdot\tau\right\rbrace = \lbrace g_\omega ^{\pm m} \rbrace_{m=0,\ldots, k}.$$
	In particular $A_\omega(k)$ has cardinality $2k + 1$. We define also the sets
	$$A_n(k)=\left\lbrace h_n \in \Gamma_n \text{ s.t. } d(x_n, h_n x_n) \leq \left(k+\frac{1}{2}\right)\cdot\tau\right\rbrace.$$
	Since we have a uniform bound on the systole and the action is torsion-free then $\#A_n(k) \leq \#A_\omega(k)$ $\omega$-a.s., for every fixed $k\in\mathbb{N}$.  We apply this property to $k_0=\frac{72\delta + 2D}{\tau}$. Clearly $g_n^{\pm m} \in A_n(k_0)$ for every $m=0,\ldots,k_0$, $\omega$-a.s. Therefore $A_n(k_0)=\lbrace g_n^{\pm m} \rbrace_{m=0,\ldots,k_0}$ $\omega$-a.s. By Lemma \ref{lemma-generating-set} we conclude that
	$$\Gamma_n = \langle A_n(k_0) \rangle = \langle g_n \rangle,$$
	i.e. $\Gamma_n$ is elementary $\omega$-a.s., which is a contradiction. 
\end{proof}

It is interesting to compare our convergence with the Gromov-Hausdorff convergence of the quotient spaces.
\begin{theo}
	Let $(X_n,x_n,\Gamma_n), (X_\infty, x_\infty, \Gamma_\infty) \in \mathcal{M}(\delta,D)$.
	\begin{itemize}
		\item[(i)] If $(X_n,x_n,\Gamma_n) \underset{\textup{eq-pGH}}{\longrightarrow} (X_\infty, x_\infty,\Gamma_\infty)$ then $(\Gamma_n\backslash X_n, \bar{x}_n) \underset{\textup{pGH}}{\longrightarrow} (\Gamma_\infty\backslash X_\infty, \bar{x}_\infty)$ and $\sup_{n\in\mathbb{N}}h_{\Gamma_n} < +\infty$;
		\item[(ii)] if $(\Gamma_n\backslash X_n, \bar{x}_n) \underset{\textup{pGH}}{\longrightarrow} (Y,y)$ and $\sup_{n\in\mathbb{N}}h_{\Gamma_n} < +\infty$ then there exists a subsequence $\lbrace n_k \rbrace$ such that $(X_{n_k},x_{n_k},\Gamma_{n_k}) \underset{\textup{eq-pGH}}{\longrightarrow} (X_\infty, x_\infty,\Gamma_\infty)$ and $(Y,y)$ is isometric to $(\Gamma_\infty\backslash X_\infty, \bar{x}_\infty)$.
	\end{itemize}  
\end{theo}
\begin{proof}
	If $(X_n,x_n,\Gamma_n) \underset{\textup{eq-pGH}}{\longrightarrow} (X_\infty, x_\infty,\Gamma_\infty)$ then $\sup_{n\in\mathbb{N}}h_{\Gamma_n} < +\infty$ by Corollary \ref{cor-entropy-bound-convergence}. The second part of the first statement is true once we show that the ultralimit $(\bar X_\omega, \bar x_\omega)$ of the sequence $(\Gamma_n\backslash X_n =:\bar X_n, \bar{x}_n)$ is isometric to $(\Gamma_\infty\backslash X_\infty, \bar{x}_\infty)$ for every non-principal ultrafilter, by Corollary \ref{cor-ultralimit-unique-limit}. We fix a non-principal ultrafilter $\omega$. By Proposition \ref{prop-GH-ultralimit} the triple $(X_\omega, x_\omega, \Gamma_\omega)$ is equivariantly isometric to $(X_\infty, x_\infty, \Gamma_\infty)$.\\
	The projections $p_n\colon X_n \to  \bar X_n$ form an admissible sequence of $1$-Lipschitz maps and then, by Proposition A.5 of \cite{CavS20}, they yield a limit map $p_\omega\colon X_\omega \to  \bar X_\omega$ defined as $p_\omega ( y_\omega) = \omega$-$\lim p_n(   y_n)$, for $\omega$-$\lim  y_n =  y_\omega$. 
	The map $p_\omega$ is clearly surjective. It is also $\Gamma_\omega$-equivariant: indeed
	$$p_\omega (\gamma_\omega   y_\omega) = \omega\text{-}\lim p_n(\gamma_n y_n) = \omega\text{-}\lim p_n(y_n) = p_\omega(y_\omega)$$
	for every $g_\omega = \omega$-$\lim g_n \in \Gamma_\omega$ and $y_\omega = \omega$-$\lim y_n \in X_\omega$.
	Therefore we have a well defined, surjective quotient map $\bar{p}_\omega\colon \Gamma_\omega\backslash X_\omega \to \bar X_\omega$.
	The next step is to show it is a local isometry. We fix an arbitrary point $y_\omega = \omega\text{-}\lim y_n \in X_\omega$ and we  consider its class $[y_\omega]  \in \Gamma_\omega\backslash X_\omega$. By Proposition \ref{prop-bound-systole} there exists $s>0$ such that sys$(\Gamma_n, X_n) \geq s$ for every $n$. So, as in the proof of Corollary \ref{cor-discrete-limit}, the systole of $\Gamma_\omega$ is at least $s$. Therefore the quotient map $X_\omega \to \Gamma_\omega\backslash X_\omega$ is an isometry between $\overline{B}(y_\omega, \frac{s}{2})$ and $\overline{B}([y_\omega],\frac{s}{2})$. Moreover $\overline{B}(p_n(y_n), \frac{s}{2})$ is isometric to $\overline{B}(y_n,\frac{s}{2})$ for every $n$. By Lemma A.8 of \cite{CavS20} we know that  $\omega$-$\lim \overline{B}(p_n(y_n), \frac{s}{2})$ is isometric to $\overline{B}(p_\omega(y_\omega), \frac{s}{2})=\overline{B}(\bar p_\omega([y_\omega]), \frac{s}{2})$ and that  $\omega$-$\lim \overline{B}(y_n, \frac{s}{2})$ is isometric to $\overline{B}(y_\omega, \frac{s}{2})$.  Therefore $\overline{B}(\bar p_\omega([y_\omega]), \frac{s}{2})$ is isometric to $\overline{B}([y_\omega],\frac{s}{2})$, i.e. $\bar p_\omega$ is a local isometry.\\
	Now we prove $\bar{p}_\omega$ is injective. Let $[z_\omega], [y_\omega] \in \Gamma_\omega\backslash X_\omega$. Clearly
	$\bar{p}_\omega([z_\omega]) = \bar{p}_\omega([y_\omega])$ if and only if $p_\omega(z_\omega)=p_\omega(y_\omega)$. This means $\omega\text{-}\lim d(p_n (z_n), p_n(y_n)) = 0$ and, as the systole of $\Gamma_n$ is $\geq s>0$,  we have $\omega$-$\lim d(z_n, g_n y_n) = 0$ for some $g_n\in \Gamma_n$, $\omega$-a.s. The sequence $(g_n)$ is admissible, hence it defines an element $g_\omega = \omega$-$\lim g_n \in \Gamma_\omega$ satisfying $d(z_\omega,g_\omega y_\omega)=0$. This implies $[z_\omega] = [y_\omega]$.\\
	The map $\bar p_\omega\colon \Gamma_\omega \backslash X_\omega \to \bar X_\omega$ is a bijective local isometry between two length spaces. If its inverse is continuous then it is an isometry. We take points $\bar y_\omega^k = \omega$-$\lim \bar y_n^k \in \bar X_\omega$ converging to $\bar y_\omega^\infty = \omega$-$\lim \bar y_n^\infty \in \bar X_\omega$ as $k \to +\infty$. We have $\bar y_n^k = p_n(y_n^k), \bar y_n^\infty = p_n(y_n^\infty)$ for some $y_n^k,y_n^\infty \in X_n$. We can suppose that $y_n^k,y_n^\infty$ belong to a fixed ball around $x_n$. We consider the points $y_\omega^k = \omega$-$\lim y_n^k$ and $y_\omega^\infty = \omega$-$\lim y_n^\infty$ of $X_\omega$ and their images $[y_\omega^k], [y_\omega^\infty] \in \Gamma_\omega \backslash X_\omega$. It is straightforward to show that $\bar p_\omega ([y_\omega^k]) = \bar y_\omega^k$ and $\bar p_\omega ([y_\omega^\infty]) = \bar y_\omega^\infty$. Now it is not difficult to check that the sequence $[y_\omega^k]$ converges to $[y_\omega^\infty]$ when $k \to +\infty$, proving that the inverse of $\bar{p}_\omega$ is continuous. \\
	Therefore $(\bar X_\omega, \bar{x}_\omega)$ is isometric to $(\Gamma_\omega \backslash X_\omega, p_\omega{x_\omega})$ which is clearly isometric to $(\Gamma_\infty \backslash X_\infty, \bar{x}_\infty)$. The proof of (i) is then finished since this is true for every non-principal ultrafilter $\omega$.
	\vspace{2mm}
	
	Suppose now that $(\Gamma_n\backslash X_n, \bar{x}_n) \underset{\textup{pGH}}{\longrightarrow} (Y,y)$ and $\sup_{n\in\mathbb{N}}h_{\Gamma_n} < +\infty$. Again by Proposition \ref{prop-bound-systole} there exists $s>0$ such that sys$(\Gamma_n, X_n) \geq s$ for every $n$. We fix a non-principal ultrafilter $\omega$ and we consider the ultralimit triple $(X_\omega,x_\omega,\Gamma_\omega)$. As usual we get sys$(\Gamma_\omega,X_\omega) \geq s$. We can apply the same argument above to show that $\Gamma_\omega\backslash X_\omega$ is isometric to $Y$. Moreover by the condition on the systole of $\Gamma_\omega$ we know that $X_\omega$ is locally isometric to $\Gamma_\omega\backslash X_\omega$.	
	Since $Y$ is compact we conclude that $X_\omega$ is a geodesic, complete, locally compact metric space. Therefore it is proper by Hopf-Rinow's Theorem, see Corollary I.3.8 of \cite{BH09}. Then there exists a subsequence $\lbrace n_k \rbrace$ such that $(X_{n_k},x_{n_k},\Gamma_{n_k}) \underset{\textup{eq-pGH}}{\longrightarrow} (X_\omega, x_\omega,\Gamma_\omega)$ by Proposition \ref{prop-GH-ultralimit}.
\end{proof}

\section{Continuity of the critical exponent}

Let $\Gamma$ be a discrete, quasiconvex-cocompact group of isometries of a proper, $\delta$-hyperbolic metric space $X$. Then it is proved in \cite{Coo93} that the Patterson-Sullivan measure on $\Lambda(\Gamma)$ is $(A,h_\Gamma)$-Ahlfors regular for some $A> 0$. Our goal is to quantify the constant $A$.
\begin{theo}
	\label{cocompactness}
	Let $\delta,D,H\geq 0$. There exists $A=A(\delta,D,H) \geq 1$ such that for all $(X,x,\Gamma)\in \mathcal{M}(\delta,D)$ with $h_\Gamma \leq H$ the subset $\Lambda(\Gamma)$ is visually $(A,h_\Gamma)$-Ahlfors regular with respect to every Patterson-Sullivan measure.
\end{theo}

\begin{proof}
	We divide the proof in steps.\\
	{\textbf{Step 1:} \em{$\forall z\in \partial X$ and $\forall\rho>0$ it holds $\mu_{\textup{PS}}(B(z,\rho))\leq e^{h_\Gamma(55\delta + 3D)}\rho^{h_\Gamma}$.}}\vspace{1mm} \\
	We suppose first $z\in \Lambda(\Gamma)$ and we take the set \vspace{-3mm}
	$$\tilde{B}(z,\rho)= \left\lbrace y \in X \cup \partial X \text{ s.t. } (y,z)_{x} > \log\frac{1}{\rho} \right\rbrace.\vspace{-3mm}$$
	It is open (cp. Observation 4.5.2 of \cite{DSU17}) and $\tilde{B}(z,\rho)\cap \partial X = B(z,\rho)$, so $\mu_{\text{PS}}(\tilde{B}(z,\rho)) = \mu_{\text{PS}}(B(z,\rho))$ since $\mu_{\text{PS}}$ is supported on $\Lambda(\Gamma)\subseteq \partial X$.
	Let $T=\log\frac{1}{\rho}$ and $\xi_{x,z}$ be a geodesic ray $[x,z]$. For every $y\in\Gamma x \cap \tilde{B}(z,\rho)$ we have
	\begin{equation}
		\label{ppp}
		d(x,y)\geq T - \delta \,\,\,\text{ and } \,\,\, d(x,y) \geq d(x,\xi_{x,z}(T)) + d(\xi_{x,z}(T), y) - 12\delta.
	\end{equation}
	The first inequality is given by Remark \ref{rmk-product-finite}. Let $\gamma$ be any geodesic segment $[x,y]$. Again by Remark \ref{rmk-product-finite} we have $d(\xi_{x,z}(T), \gamma(T)) \leq 6\delta$, therefore
	$$d(x,y) = d(x,\gamma(T)) + d(\gamma(T), y) \geq d(x,\xi_{x,z}(T)) + d(\xi_{x,z}(T), y) -12\delta.$$
	Moreover we have $d(\xi_{x,z}(T), \text{QC-Hull}(\Lambda(\Gamma))) \leq 14\delta$ by Lemma \ref{approximation-ray-line}, since $x\in \text{QC-Hull}(\Lambda(\Gamma))$. 
	By the cocompactness of the action on $\text{QC-Hull}(\Lambda(\Gamma))$ we can find a point $x_1 \in \Gamma x$ such that $d(\xi_{x,z}(T), x_1)\leq 14\delta+D$. This implies 
	$$d(x,y) \geq d(x,x_1) + d(x_1, y) - 40\delta - 2D$$
	for every $y\in \Gamma x \cap \tilde{B}(z,\rho)$.
	Therefore
	\begin{equation*}
		\begin{aligned}
			\sum_{y \in \Gamma x \cap \tilde{B}(z,\rho)} e^{-s d(x,y)} &\leq \sum_{y \in \Gamma x \cap \tilde{B}(z,\rho)} e^{-s (d(x, x_1) + d(x_1,y) - 40 \delta - 2D)} \\
			&= e^{s(40\delta + 2D)}e^{-sd(x,x_1)}\cdot \sum_{y \in \Gamma x \cap \tilde{B}(z,\rho)}e^{-sd(x_1,y)} \\
			&\leq e^{s(54\delta + 3D)}e^{-sd(x,\xi_{xz}(T))}\cdot  \sum_{g \in \Gamma}e^{-sd(x_1,g x_1)} \\
			&= e^{s(54\delta + 3D)}\cdot\rho^s\cdot \sum_{g \in \Gamma}e^{-sd(x,g x)}.
		\end{aligned}
	\end{equation*}
	In other words we have $\mu_s(\tilde{B}(z,\rho)) \leq e^{s(54\delta + 3D)}\rho^s$, and by $\ast$-weak convergence we conclude that
	\vspace{-2mm}
	\begin{equation*}
		\begin{aligned}
			\mu_{\text{PS}}(B(z,\rho)) = \mu_{\text{PS}}(\tilde{B}(z,\rho)) \leq \liminf_{i\to +\infty} \mu_{s_i}(\tilde{B}(z,\rho)) \leq e^{h_\Gamma(54\delta + 3D)}\rho^{h_\Gamma}.
		\end{aligned}
		\vspace{-2mm}
	\end{equation*}
	If $z\in \partial X$ we observe that if $B(z,\rho )\cap \Lambda(\Gamma)=\emptyset$ then the thesis is obviously true since $\mu_{\text{PS}}$ is supported on $\Lambda(\Gamma)$. Otherwise there exists $w\in \Lambda(\Gamma)$ such that $(z,w)_{x} > \log\frac{1}{\rho}$. It is straightforward to check that $B(w,\rho) \subseteq B(z,\rho e^\delta)$ by \eqref{hyperbolicity-boundary}. Then $\mu_{\text{PS}}(B(z,\rho))\leq e^{h_\Gamma(55\delta + 3D)}\rho^{h_\Gamma}$. \vspace{4mm}
	\\
	{\textbf{Step 2:} \em{for every $R\geq R_0 :=\frac{\log 2}{h_\Gamma} + 55\delta + 3D + 5\delta$ and for every $g \in \Gamma$ it holds
			$\mu_{\textup{PS}}(\textup{Shad}_{x}(g x, R)) \geq \frac{1}{2Q}e^{-h_\Gamma d(x,g x)},$ where $Q$ is the constant of Proposition \ref{patterson-sullivan} that depends only on $\delta$ and $H$.}} \vspace{1mm}\\
	From the first step we know that for every $\rho \leq \rho_0 := 2^{-\frac{1}{h_\Gamma}}e^{-(55\delta + 3D)}$ and for every $z\in \partial X$ it holds $\mu_{\text{PS}}(B(z,\rho))\leq \frac{1}{2}.$ A direct computation shows that $R_0 = \log \frac{1}{\rho_0} + 5\delta$. We claim that for every $R\geq R_0$ and every $g \in \Gamma$ the set
	$\partial X \setminus g(\text{Shad}_{x}(g^{-1}x,R))$ is contained in a generalized visual ball of radius at most $\rho_0$. Indeed if $z,w \in \partial X \setminus g(\text{Shad}_{x}(g^{-1}x,R))$ then there are geodesic rays $\xi =[g x, z], \xi'=[g x, w]$ that do not intersect the ball $B(x,R)$. Therefore we get $(\xi(T),g x)_{x} \geq d(x, [g x, \xi(T)]) - 4\delta \geq R - 4\delta$ by Lemma \ref{projection}, so $(z,g x)_{x} \geq \liminf_{T\to +\infty} (\xi(T),g x)_{x} \geq R - 4\delta$. The same holds for $w$. Thus by \eqref{hyperbolicity-boundary} we get $(z,w)_{x} \geq R - 5\delta,$ proving the claim.
	By Proposition \ref{patterson-sullivan} we get
	$$\frac{\mu_{\text{PS}}(\text{Shad}_{x}(g x, R))}{\mu_{\text{PS}}(g^{-1}(\text{Shad}_{x}(g x, R)))} \geq \frac{1}{Q}e^{-h_\Gamma (B_z(x,x) - B_z(x,g^{-1} x))}.$$
	Since $R\geq R_0$ the measure of $g^{-1}(\text{Shad}_{x}(g x, R))$ is at least $\frac{1}{2}$. Moreover the Busemann function is $1$-Lipschitz, so
	$$\mu_{\text{PS}}(\text{Shad}_{x}(g x, R)) \geq \frac{1}{2Q}e^{-h_\Gamma d(x,g^{-1}x)} = \frac{1}{2Q}e^{-h_\Gamma d(x,gx)}.$$
	{\textbf{Step 3.} \em{$\mu_{\textup{PS}}(B(z,\rho))\geq \frac{1}{2Q}e^{-h_\Gamma(R_0+28\delta+2D)}\rho^{h_\Gamma}$ for every $z\in \Lambda(\Gamma)$ and every $\rho>0$.}}\vspace{1mm}\\
	For every $\rho > 0$ we set $T=\log\frac{1}{\rho}$. If $z\in \partial X$ and $R\geq 0$ then by Lemma \ref{shadow-ball} we get Shad$_{x}(\xi_{x,z}(T+R), R) \subseteq B(z,e^{-T})$. 
	We take $R=R_0 + 14\delta + D$, where $R_0$ is the constant of the second step and we conclude that Shad$_{x}(\xi_{x,z}(T+R), R)$ is contained in $B(z,\rho)$.
	Applying again Lemma \ref{approximation-ray-line} and the cocompactness of the action we can find $g \in \Gamma$ such that $d(gx, \xi_{x,z}(T+R))\leq 14\delta + D$, implying Shad$_{x}(gx,R_0)\subseteq \text{Shad}_{x}(\xi_{x,z}(T+R), R) \subseteq B(z,\rho)$.  
	From the second step we obtain $\mu_{\text{PS}}(B(z,\rho)) \geq \frac{1}{2Q}e^{-h_\Gamma d(x,g x)}.$
	Furthermore $d(x, g x) \leq T + R_0 + 28\delta + 2D,$
	so finally
	$$\mu_{\text{PS}}(B(z,\rho))\geq \frac{1}{2Q}e^{-h_\Gamma(R_0+28\delta +2D)}\rho^{h_\Gamma}.$$	
	The explicit description of the constants shows as they depend only on $\delta,H,D$ and on a lower bound on $h_\Gamma$, which is given in terms of $\delta$ and $D$ by Proposition \ref{prop-lower-bound-entropy}.
\end{proof}

As a consequence, applying Corollary \ref{cor-entropy-bound-convergence}, we have
\begin{cor}
	In the standard setting of convergence there exists some $A>0$ such that every visual boundary $\partial X_n$ is visually $(A,h_{\Gamma_n})$-Ahlfors-regular with respect to any Patterson-Sullivan measure.
\end{cor}
With this result it is possible to show the continuity of the critical exponent under the standard setting of convergence. However we prefer to use the equidistribution of the orbits, following again the ideas of \cite{Coo93}.


\begin{proof}[Proof of Theorem \ref{theo-uniform-distribution}]
	By Corollary \ref{cor-entropy-bound-convergence} we have $\sup_{n\in\mathbb{N}}h_{\Gamma_n} =: H < + \infty$. By Proposition \ref{prop-bound-systole} there exists $s > 0$ such that $\text{sys}(\Gamma_n,X_n) \geq s$ for every $n$. 
	Let $R_0 = R_0(\delta, D, H)$ be the number of Step 2 of Theorem \ref{cocompactness} and $Q$ be the constant of Proposition \ref{patterson-sullivan}. 
	By Lemma \ref{lemma-uniform-packing-convergence} we have $$\sup_{n\in\mathbb{N}}\text{Pack}_{Y_n}\left(4R_0+1,\frac{s}{2}\right) =: N <+\infty.$$
	We fix $n \in \mathbb{N}$. It is easy to check that if $[x_n,z_n]\cap B(y_n,R_0) \neq \emptyset$ and $[x_n,z_n]\cap B(y'_n,R_0) \neq \emptyset$, where $z_n\in \partial X$ and $y_n,y'_n$ are points of $X_n$ with $\vert d(x_n,y_n) - d(x_n,y'_n)\vert \leq 1$, then $d(y_n,y'_n)\leq 4R_0 + 1$. Thus for every $j\in \mathbb{N}$ we have $\#\lbrace y_n\in \Gamma_n x_n \text{ s.t. } y_n\in A(x_n,j,j+1) \text{ and } z_n\in \textup{Shad}_{x_n}(y_n,R_0) \rbrace \leq N.$\vspace{1mm}\\
	{\noindent\textbf{Step 1. }{\em For all $k\in \mathbb{N}$ it holds
			$\#\Gamma_n x_n \cap \overline{B}(x_n,k) \leq 4QNe^{h_{\Gamma_n} k}.$ 
	}}\vspace{1mm}\\
	Let $A_{n,j} = \Gamma_n x_n \cap A(x_n,j,j+1)$. By the observation made before we conclude that among the set of shadows $\lbrace \text{Shad}_{x_n}(y_n,R_0) \rbrace_{y_n\in A_j}$ there are at least $\frac{\# A_j}{N}$ disjoint sets. Thus
	$$1 \geq \mu_{\text{PS}}\left( \bigcup_{y_n\in A_{n,j}}\text{Shad}_{x_n}(y_n,R_0) \right) \geq \frac{\# A_{n,j}}{N}\cdot \frac{1}{2Q}e^{-h_{\Gamma_n}(j+1)},$$
	where we used Step 2 of Theorem \ref{cocompactness}. This implies $\#A_{n,j} \leq 2QNe^{h_{\Gamma_n}(j+1)}$ for every $j\in \mathbb{N}$. Finally we have
	$$\#\Gamma_n x_n \cap \overline{B}(x_n,k) \leq \sum_{j=0}^{k-1} \#A_{n,j} \leq 4QNe^{h_{\Gamma_n} k}.$$
	{\noindent\textbf{Step 2. }{\em For all $T\geq 0$ it holds
			$\#\Gamma_n x_n \cap \overline{B}(x_n,T) \geq e^{-h_{\Gamma_n}(84\delta + 5D + 1)}e^{h_{\Gamma_n} T}.$ 
	}}\vspace{1mm}\\
	We fix $z_1^n,\ldots,z_{K_n}^n \in \Lambda(\Gamma_n)$ realizing Pack$^*(\Lambda(\Gamma_n), e^{-T + 28\delta + 2D + 1})$: in particular $(z_i^n,z_j^n)_{x_n} \leq T - 28\delta - 2D - 1$ for all $1\leq i \neq j \leq K_n$. By Lemma \ref{product-rays} we deduce that $d(\xi_{x_n,z_i^n}(T-14\delta -D), \xi_{x_n,z_j^n}(T-14\delta -D)) \geq 28\delta + 2D + 1 > 28\delta + 2D$. Moreover for every $1\leq i \leq K_n$ we can find a point $y_i^n\in \Gamma x$ such that $d(\xi_{x_n,z_i^n}(T-14\delta -D), y_i^n) \leq 14\delta + D$ by Lemma \ref{approximation-ray-line}. Therefore we have $d(x_n,y_i^n)\leq T$ and $d(y_i^n,y_j^n) > 0$ for every $1\leq i \neq j \leq K_n$. So
	\begin{equation*}
		\begin{aligned}
			\#\Gamma_n x_n \cap \overline{B}(x_n,T) &\geq \text{Pack}^*(\Lambda(\Gamma_n), e^{-T+28\delta + 2D + 1}) \\
			&\geq \text{Cov}(\Lambda(\Gamma_n), e^{-T+29\delta + 2D + 1}) \\
			&\geq e^{-h_{\Gamma_n}(84\delta + 5D + 1)}e^{h_{\Gamma_n} T}.
		\end{aligned}
	\end{equation*}
	The first inequality follows from the discussion above, while the second one is Lemma \ref{packing*}. The last inequality follows by Step 1 of Theorem \ref{cocompactness}. Indeed we get
	\begin{equation*}
		\begin{aligned}
			\text{Cov}(\Lambda(\Gamma_n), e^{-T + 29\delta + 2D + 1}) &\geq e^{-h_{\Gamma_n}(55\delta + 3D)}e^{-h_{\Gamma_n}(-T + 29\delta + 2D + 1)} \\ &=e^{-h_{\Gamma_n}(84\delta + 5D + 1)}e^{h_{\Gamma_n} T}.
		\end{aligned}
	\end{equation*}
	The thesis follows by the bounded quantification of all the constants involved in terms of $\delta, D, H, N$ and the lower bound on the critical exponent given in terms of $\delta$ and $D$ by Proposition \ref{prop-lower-bound-entropy}.
\end{proof}

We can conclude now the
\begin{proof}[Proof of Theorem \ref{theorem-main}.(ii)]
	Let $\omega$ be a non-principal ultrafilter.
	By Proposition \ref{prop-GH-ultralimit} the triple $(X_\omega,x_\omega, \Gamma_\omega)$ is equivariantly isometric to $(X_\infty, x_\infty, \Gamma_\infty)$. By Theorem \ref{theorem-main}.(i) the triple $(X_\omega,x_\omega, \Gamma_\omega)$ belongs to $\mathcal{M}(\delta,D)$. By Proposition \ref{prop-critical-exponent-covering-entropy} the critical exponent of $\Gamma_\omega$ is finite, so for every $\varepsilon > 0$ there exists $T_\varepsilon \geq 0$ such that if $T\geq T_\varepsilon$ then 
	\begin{equation}
		\label{eq-limit-exponent}
		e^{T(h_{\Gamma_\omega} -\varepsilon)} \leq  \#\Gamma_\omega x_\omega\cap\overline{B}(x_\omega,T) \leq e^{T(h_{\Gamma_\omega} +\varepsilon)},
	\end{equation} 
	by Lemma \ref{entropy-critical}. We fix $K$ as in Theorem \ref{theo-uniform-distribution} and we set $T:= \max\lbrace T_\varepsilon, \frac{\log{K\cdot e}}{\varepsilon}\rbrace$. It is not difficult to show that
	$$\#\Gamma_nx_n\cap\overline{B}(x_n,T-1) \leq \#\Gamma_\omega x_\omega\cap\overline{B}(x_\omega,T) \leq \#\Gamma_nx_n\cap\overline{B}(x_n,T+1)$$
	$\omega$-a.s., so
	\begin{equation}
		\label{eq-limit-condition}
		\frac{1}{K}\cdot e^{-1}\cdot e^{T\cdot h_{\Gamma_n}} \leq \#\Gamma_\omega x_\omega\cap\overline{B}(x_\omega,T)  \leq K \cdot e \cdot e^{T\cdot h_{\Gamma_n}}
	\end{equation}
	$\omega$-a.s. Putting together \eqref{eq-limit-exponent} and \eqref{eq-limit-condition} and using the definition of $T$ we get
	$$h_{\Gamma_n} - 2\varepsilon \leq h_{\Gamma_\omega} \leq h_{\Gamma_n} + 2\varepsilon,$$
	$\omega$-a.s. This means $\omega$-$\lim h_{\Gamma_n} = h_{\Gamma_\omega}$, by definition. This is true for every non-principal ultrafilter, hence the continuity under equivariant pointed Gromov-Hausdorff convergence follows by the following lemma.
\end{proof}

\begin{lemma}
	Let $a_n$ be a bounded sequence of real numbers. 
	\begin{itemize}
		\item[(i)] If $a_{n_j}$ is a subsequence converging to $\tilde{a}$ then there exists a non-principal ultrafilter $\omega$ such that $\omega$-$\lim a_n = \tilde{a}$;
		\item[(ii)] if there exists $a\in \mathbb{R}$ such that $\omega$-$\lim a_n = a$ for every non-principal ultrafilter $\omega$, then $\exists \lim_{n\to +\infty} a_n = a$.
	\end{itemize} 
\end{lemma}
\begin{proof}
	Let us start with (i). The set $\{n_j\}_j$ is infinite, then there exists a non-principal ultrafilter $\omega$ containing $\{n_j\}_j$ (cp. \cite{Jan17}, Lemma 3.2). Moreover for every $\varepsilon > 0$ there exists $j_\varepsilon$ such that for all $j\geq j_\varepsilon$ it holds $\vert a_{n_j} - \tilde{a} \vert < \varepsilon$. The set of indices where the inequality is true belongs to $\omega$ since the complementary is finite. This implies exactly that $\tilde{a} = \omega$-$\lim a_n$.\\
	The proof of (ii) is now a direct consequence: we take subsequences $\lbrace n_j\rbrace_{j\in J}$, $\lbrace n_k\rbrace_{k\in K}$ converging respectively to the limit inferior and limit superior of the sequence. By (i) there are two non-principal ultrafilters $\omega_J$, $\omega_K$ such that $\omega_J$-$\lim a_{n_j} = \liminf_{n\to+\infty}a_n$ and $\omega_K$-$\lim a_{n_k} = \limsup_{n\to+\infty}a_n$. By assumption these two ultralimits coincide, so $\liminf_{n\to+\infty}a_n = \limsup_{n\to+\infty}a_n$.
\end{proof}

\section{Algebraic and equivariant GH-convergence}
\label{subsec-algebraic}
Let $X$ be a proper metric space and $G$ be a topological group. We denote by Act$(G,X)$ the set of homomorphisms $\varphi \colon G \to \text{Isom}(X)$.
\begin{defin}
	Let $\varphi_n,\varphi_\infty \in \textup{Act}(G,X)$. We say $\varphi_n$ converges in the algebraic sense to $\varphi_\infty$ if $\varphi_n$ converges to $\varphi_\infty$ in the compact-open topology. In this case we write $\varphi_n \underset{\textup{alg}}{\longrightarrow} \varphi_\infty$. The compact-open topology is Hausdorff since Isom$(X)$ is, so the algebraic limit is unique, if it exists.\\
	If $G$ has the discrete topology then the algebraic convergence is equivalent to: for every $g\in G$ the isometries $\varphi_n(g)$ converge to $\varphi_\infty(g)$ uniformly on compact subsets of $X$. \\	
	\noindent If $G$ has the discrete topology and is finitely generated by $\lbrace g_1,\ldots,g_\ell\rbrace$ then the algebraic convergence is equivalent to: for every $i=1,\ldots,\ell$ the isometries $\varphi_n(g_i)$ converge to $\varphi_\infty(g_i)$ uniformly on compact subsets of $X$. 
\end{defin}
The algebraic limit is always contained in the ultralimit group in the following sense.
\begin{prop}
	Let $\varphi_n,\varphi_\infty \in\textup{Act}(G,X)$ and suppose $\varphi_n \underset{\textup{alg}}{\longrightarrow}{\varphi_\infty}$. Let $\omega$ be a non-principal ultrafilter. Then $\varphi_\infty(G) \subseteq (\varphi_n(G))_\omega$, where we use Lemma \ref{lemma-ultralimit-constant} to identify the ultralimit group of the sequence $\varphi_n(G)$ to a group of isometries of $X$.
\end{prop}
\begin{proof}
	For every $g\in G$ the sequence $\varphi_n(g)$ converges uniformly on compact subsets of $X$ to $\varphi_\infty(g)$ by assumption. It is easy to see that the ultralimit element $\omega$-$\lim \varphi_n(g)$ coincides with $\varphi_\infty(g)$, so $ \varphi_\infty(G) \subseteq (\varphi_n(G))_\omega$.
\end{proof}

In general the inclusion is strict.
\begin{ex}
	\label{ex-algebraic-GH}
	Let $X=\mathbb{R}, G = \mathbb{Z}$ and $\varphi_n \colon \mathbb{Z} \to \textup{Isom}(\mathbb{R})$ defined by sending $1$ to the translation of length $\frac{1}{n}$. Clearly the sequence $\varphi_n$ converges algebraically to the trivial homomorphism. On the other hand $(\varphi_n(\mathbb{Z}))_\omega$ is the group $\Gamma$ of all translations of $\mathbb{R}$ for every non-principal ultrafilter $\omega$.
\end{ex}

However the two limits coincide when restricted to the class $\mathcal{M}(\delta,D)$.

\begin{theo}
	\label{theo-ultralimit-to-algebraic}
	Let $(X,x, \Gamma_n)\in \mathcal{M}(\delta,D)$.
	\begin{itemize}
		\item[(i)] if $(X,x, \Gamma_n) \underset{\textup{eq-pGH}}{\longrightarrow}(X,x, \Gamma_\infty)$ then there exists a group $G$ such that for every $n$ big enough $\Gamma_{n} = \varphi_{n}(G)$, $\Gamma_\infty = \varphi_\infty(G)$ with $\varphi_{n},\varphi_\infty$ isomorphisms and $\varphi_{n} \underset{\textup{alg}}{\longrightarrow} \varphi_\infty$.
		\item[(ii)] Viceversa if there exist a group $G$ and isomorphisms $\varphi_n\colon G \to \Gamma_n$ and if $\varphi_n \underset{\textup{alg}}{\longrightarrow} \varphi_\infty$ then $(X,x,\Gamma_n) \underset{\textup{eq-pGH}}{\longrightarrow} (X,x,\varphi_\infty(G))$.
	\end{itemize}
\end{theo}



\noindent Before the proof we need two results. Given a group $\Gamma$ and a finite set of generators $\Sigma$ of $\Gamma$ it is classically defined the word-metric $d_\Sigma$ on $\Gamma$ as 
$$d_\Sigma(g,h) := \inf\lbrace \ell \in \mathbb{N} \text{ s.t. } g =  h \cdot \sigma_1\cdots\sigma_\ell, \text{ where } \sigma_i\in \Sigma\rbrace.$$
By definition $d_\Sigma$ takes value in the set of natural numbers and $d_\Sigma(g,h)=d_\Sigma(h^{-1}g, \text{id})$. The couple $(\Gamma,\Sigma)$ is called a marked group.
\begin{lemma}[Lemma 4.6 of \cite{BCGS2}]
	\label{lemma-word-metric-comparison}
	Let $(X,x,\Gamma)\in\mathcal{M}(\delta,D)$ and let $R > 2D + 72\delta$. Set $\Sigma:=\Sigma_R(\Gamma,x)$  and denote by $d_\Sigma$ the associated word-metric on $\Gamma$ (observe that $\Sigma$ is a generating set of $\Gamma$ by Lemma \ref{lemma-generating-set}). Then 
	$$(R - 2D - 72\delta) \cdot d_\Sigma(g,h) \leq d(gx,hx) \leq  R \cdot d_\Sigma(g,h)$$
	for all $g,h\in \Gamma$.
\end{lemma}
\begin{proof}
	It is enough to check the inequalities for $g\in \Gamma$ and $h=\text{id}$. We write $g=\sigma_1\cdots\sigma_\ell$, with $\sigma_i \in \Sigma$, $\ell = d_\Sigma(g,\text{id})$. The right inequality follows by the triangle inequality on $X$, indeed $d(gx,x)\leq \ell\cdot R$.\\
	We now take consecutive points $x_i$, $i=0,\ldots,\ell$, along a geodesic segment $[x,gx]$ with $x_0=x$, $x_\ell = gx$, $d(x_{i-1},x_{i}) = R-2D-72\delta$ for $i=1,\ldots,\ell -1$ and $d(x_{\ell -1}, gx) \leq R-2D-72\delta$. This implies $\ell \leq \frac{d(x,gx)}{R-2D-72\delta}$. By Lemma \ref{lemma-quasigeodesic} and by cocompactness we can find an element $g_i \in \Gamma$ such that $d(x_i,g_ix)\leq 36\delta + D$ for every $i=0,\ldots,\ell$. We choose $g_0=\text{id}$ and $g_\ell = g$. Clearly $d(g_ix, g_{i-1}x) \leq R$ for every $i=1,\ldots,\ell$. This shows that $\sigma_i = g_{i-1}^{-1}g_i \in \Sigma$. Moreover $g = \sigma_1\cdots \sigma_\ell$, i.e. $d_{\Sigma}(g,\text{id}) \leq \ell \leq \frac{d(x,gx)}{R-2D-72\delta}$.
\end{proof}

In the following proposition we explicit the metric in the pointed Gromov-Hausdorff convergence in order to make it more understable.

\begin{prop}
	\label{prop-stability-word-metrics}
	In the standard setting of convergence let $R$ be a real number satisfying
	\begin{itemize}
		\item[(i)] $2D+72\delta < R \leq 2D+72\delta +1$;
		\item[(ii)] for every $g\in \Gamma_\infty$ such that $d(x_\infty,gx_\infty)\leq R$ then $d(x_\infty, gx_\infty) < R$.
	\end{itemize}
	Let $\Sigma_n := \Sigma_R(\Gamma_n,x_n)$ and $\Sigma_\infty = \Sigma_R(\Gamma_\infty, x_\infty)$ be generating sets of $\Gamma_n$ and $\Gamma_\infty$ respectively (by Lemma \ref{lemma-generating-set}). Equip $\Gamma_n$ and $\Gamma_\infty$ with the word-metrics $d_{\Sigma_n}$, $d_{\Sigma_\infty}$ respectively. Then $(\Gamma_n, d_{\Sigma_n}, \textup{id}) \underset{\textup{pGH}}{\longrightarrow} (\Gamma_\infty, d_{\Sigma_\infty}, \textup{id}).$
\end{prop}
\begin{proof}
	We fix a non-principal ultrafilter $\omega$, we take the ultralimit triple $(X_\omega,x_\omega, \Gamma_\omega)$ and we set $\Sigma_\omega := \Sigma_R(\Gamma_\omega,x_\omega)$. $(X_\omega, x_\omega, \Gamma_\omega)$ is equivariantly isometric to $(X_\infty,x_\infty,\Gamma_\infty)$ by Proposition \ref{prop-GH-ultralimit}, so $(\Gamma_\omega, d_{\Sigma_\omega}, \text{id})$ is isometric to $(\Gamma_\infty, d_{\Sigma_\infty}, \text{id})$, and they are proper. If we show that the ultralimit of the sequence of spaces $(\Gamma_n, d_{\Sigma_n}, \text{id})$ is isometric to $(\Gamma_\omega, d_{\Sigma_\omega}, \text{id})$ we conclude by Corollary \ref{cor-ultralimit-unique-limit} that $(\Gamma_n, d_{\Sigma_n}, \textup{id}) \underset{\textup{pGH}}{\longrightarrow} (\Gamma_\infty, d_{\Sigma_\infty}, \textup{id})$. \\
	We denote by $\omega$-$\lim (\Gamma_n, d_{\Sigma_n}, \text{id})$ the ultralimit space of this sequence. Observe that each element is represented by a sequence $(g_n)$ with $g_n\in \Gamma_n$ and $d_{\Sigma_n}(\text{id},g_n) \leq M$ $\omega$-a.s., for some $M$. We define $\Phi \colon \omega$-$\lim (\Gamma_n, d_{\Sigma_n}, \text{id}) \to (\Gamma_\omega, d_{\Sigma_\omega}, \text{id})$ by sending a point $(g_n)$ to the isometry of $\Gamma_\omega$ defined by $\omega$-$\lim g_n$. We have to show that $\Phi$ is well defined. First of all the condition $d_{\Sigma_n}(\text{id},g_n) \leq M$ implies $d(x_n,g_nx_n)\leq R\cdot M$ $\omega$-a.s. by Lemma \ref{lemma-word-metric-comparison}, so the sequence $(g_n)$ is admissible and defines a limit isometry belonging to $\Gamma_\omega$. Suppose $(h_n)$ is another sequence of isometries of $\Gamma_n$ such that $\omega$-$\lim d_{\Sigma_n}(g_n,h_n) = 0$. Then $d_{\Sigma_n}(g_n,h_n) < 1$ $\omega$-a.s, thus $d_{\Sigma_n}(g_n,h_n) = 0$ $\omega$-a.s., i.e. $g_n = h_n$ $\omega$-a.s. This shows that the map $\Phi$ is well defined.\\
	It remains to show it is an isometry. Let $(g_n), (h_n) \in \omega$-$\lim (\Gamma_n, d_{\Sigma_n}, \text{id})$ and set $\ell = \omega$-$\lim d_{\Sigma_n}(g_n,h_n)$. Since any word-metric takes values in $\mathbb{N}$ we get $d_{\Sigma_n}(g_n,h_n) = \ell$ $\omega$-a.s. For these indices we can write $g_n = h_n \cdot a_n^1 \cdots a_n^\ell$ with $a_n^1,\ldots,a_n^\ell \in \Sigma_n$. The sequence of isometries $(a_n^i)$ are admissible by definition, so they define elements $a_\omega^i \in \Sigma_\omega$. We have $g_\omega = h_\omega \cdot a_\omega^1 \cdots a_\omega^\ell$, showing that $d_{\Sigma_\omega}(g_\omega,h_\omega) \leq \ell = \omega$-$\lim d_{\Sigma_n}(g_n,h_n)$. \\
	We now take isometries $g_\omega =\omega$-$\lim g_n$, $h_\omega = \omega$-$\lim h_n$ of $\Gamma_\omega$. By definition $d(x_n,g_nx_n) \leq M$ and $d(x_n,h_nx_n) \leq M$ $\omega$-a.s., for some $M$. By Lemma \ref{lemma-word-metric-comparison} we get $d_{\Sigma_n}(g_n,\text{id}) \leq M'$ and $d_{\Sigma_n}(h_n,\text{id}) \leq M'$ $\omega$-a.s., for $M' = \frac{M}{R-2D-72\delta}$. Therefore the sequences $(g_n),(h_n)$ defines point of $\omega$-$\lim (\Gamma_n, d_{\Sigma_n}, \text{id})$. Observe that this shows also that $\Phi$ is surjective. We set $d_{\Sigma_\omega}(g_\omega,h_\omega) = \ell$. Then we can write $g_\omega = h_\omega \cdot a_\omega^1\cdots a_\omega^\ell$, for some $a_\omega^i = \omega$-$\lim a_n^i \in \Sigma_\omega$. This means that $d(x_\omega, a_\omega^i x_\omega) \leq R$, so $d(x_\omega, a_\omega^i x_\omega) < R$ by our assumptions on $R$. Therefore the following finite set of conditions hold $\omega$-a.s.: $d(x_n, a_n^i x_n) \leq R$ for every $i=1,\ldots,\ell$, i.e. $a_n^i \in \Sigma_n$ $\omega$-a.s. Now observe that if $g_n \neq h_n\cdot a_n^1\cdots a_n^\ell =: b_n$ $\omega$-a.s. then $d(g_nx_n,b_nx_n)\geq s > 0$ $\omega$-a.s. Indeed by Corollary \ref{cor-entropy-bound-convergence} and Proposition \ref{prop-bound-systole} it is enough to take $s$ smaller than a uniform lower bound of the systole of all the groups $\Gamma_n$. Hence we get $d(g_\omega x_\omega, b_\omega x_\omega) > 0$ which is clearly false. This shows that $\omega$-$\lim d_{\Sigma_n}(g_n,h_n)\leq \ell$, i.e. $d_{\Sigma_\omega}(g_\omega, h_\omega) \geq \omega$-$\lim d_{\Sigma_n}(g_n,h_n)$. Therefore $\Phi$ is an isometry.
\end{proof}

\begin{proof}[Proof of Theorem \ref{theo-ultralimit-to-algebraic}.]
	We prove first (i). We can always find $R$ as in the assumptions of Proposition \ref{prop-stability-word-metrics} since $\Gamma_\infty$ is discrete by Corollary \ref{cor-discrete-limit}. So, with the same notation of Proposition \ref{prop-stability-word-metrics}, $(\Gamma_n, d_{\Sigma_n}, \textup{id}) \underset{\textup{pGH}}{\longrightarrow} (\Gamma_\infty, d_{\Sigma_\infty}, \textup{id}).$ Applying word by word the proof of Theorem 4.4 of \cite{BCGS2}, using Lemma \ref{lemma-word-metric-comparison} instead of Lemma 4.6 therein, we get only a finite number of isomorphic types of the marked groups $(\Gamma_n, \Sigma_n)$. This implies that for $n$ big enough all the marked groups $(\Gamma_n, \Sigma_n)$ are pairwise isomorphic, and in particular isomorphic to $(\Gamma_\infty, \Sigma_\infty)$. We set $G=\Gamma_\infty$, $\varphi_\infty = \text{id}$ and $\varphi_n'$ one fixed marked isomorphism between $(\Gamma_\infty, \Sigma_\infty)$ and $(\Gamma_n, \Sigma_n)$. By Corollary \ref{cor-entropy-bound-convergence} and Proposition \ref{prop-bound-systole} we can find $s>0$ such that sys$(\Gamma_n,X)\geq 2s$ for every $n$.	
	By definition of equivariant pointed Gromov-Hausdorff convergence for each element $g\in \Sigma_\infty$ there exists $g_n\in \Gamma_n$ such that $d(g_n x, gx) < s$, if $n$ is big enough. By the condition on the systole we deduce that $g_n$ is unique. Finally by the definition of $R$, if $n$ is taken maybe bigger, every such $g_n$ belongs to $\Sigma_n$. Clearly this correspondance $g \mapsto g_n$ is one-to-one. This means that there exists a permutation $\mathcal{P}_n$ of the set $\Sigma_n$ such that $\varphi_n = \mathcal{P}_n \circ \varphi_n'$ is again a marked isomorphism between $(\Gamma_\infty, \Sigma_\infty)$ and $(\Gamma_n,\Sigma_n)$ such that $\varphi_n(g)=g_n$.
	It is now trivial to show that $\varphi_n(g)$ converges uniformly on compact subsets of $X$ to $g$ for every $g\in \Sigma_\infty$, and so that $\varphi_n$ converges algebraically to $\varphi_\infty$.

	We show now (ii). By Corollary \ref{cor-ultralimit-unique-limit} it is enough to show that $\Gamma_\omega = \varphi_\infty(G)$ for every non-principal ultrafilter $\omega$, where $\Gamma_\omega$ is the ultralimit group of the sequence $\Gamma_n$. Here we are using Lemma \ref{lemma-ultralimit-constant} to identify the ultralimit group of the sequence $(X,x,\Gamma_n)$ to a group of isometries of $X$.	
	By Proposition \ref{prop-GH-ultralimit} we know that there is a subsequence $\lbrace n_k \rbrace$ such that $(X,x, \Gamma_{n_k}) \underset{\textup{eq-pGH}}{\longrightarrow}(X,x, \Gamma_\omega)$ because $X$ is proper. Therefore by the first part of the theorem there exists a homomorphism $\psi \colon G \to \text{Isom}(X)$ with $\psi(G)=\Gamma_\omega$ and $\varphi_{n_k}\underset{\text{alg}}{\longrightarrow} \psi$. So $\psi = \varphi_\infty$ by the uniqueness of the algebraic limit. This shows that $\Gamma_\omega = \varphi_\infty(G)$ and concludes the proof.
\end{proof}
We observe that the first part of the argument above shows
\begin{cor}
	\label{cor-isomorphic-constant}
	In the standard setting of convergence the groups $\Gamma_n$ are eventually isomorphic to $\Gamma_\infty$.
\end{cor}

\section{Examples}
\label{subsec-counterexamples}

We show that each assumption on the class $\mathcal{M}(\delta,D)$ is necessary in order to have the discreteness of the limit group.
\begin{ex}[Non-elementarity of the group.] We take $X_n=\mathbb{R}$, $x_n=0$ and $\Gamma_n = \mathbb{Z}_{\frac{1}{n}}$, the group generated by the translation of length $\frac{1}{n}$. It is easy to show that $(X_\omega, x_\omega)=(\mathbb{R},0)$ and $\Gamma_\omega$ is the group of all translations of $\mathbb{R}$, for every non-principal ultrafilter $\omega$. Clearly $\Gamma_\omega$, and therefore any possible limit under equivariant pointed Gromov-Hausdorff convergence, is not discrete. Observe that each $X_n$ is a proper, geodesic, $0$-hyperbolic metric space and each $\Gamma_n$ is discrete, torsion-free and cocompact with codiameter $\leq \frac{1}{n}$.
\end{ex}
\begin{ex}[Non-uniform bound on the diameter.] For every $n$ we take a genus $2$ hyperbolic surface with systole equal to $ \frac{1}{n}$. Its fundamental group acts cocompactly on $X_n=\mathbb{H}^2$ as a subgroup $\Gamma_n$ of isometries. We take a basepoint $x_n \in \mathbb{H}^2$ which belongs to the axis of an isometry of $\Gamma_n$ with translation length $\frac{1}{n}$. As in the example above $\Gamma_\omega$ contains all the possible translations along an axis of $X_\omega = \mathbb{H}^2$ and so it is not discrete, for every non-principal ultrafilter $\omega$. Observe that each $X_n$ is a proper, geodesic, $\log 3$-hyperbolic metric space and each $\Gamma_n$ is discrete, non-elementary, torsion-free and cocompact. However the codiameter of $\Gamma_n$ is not uniformly bounded above.	
\end{ex}
\begin{ex}[Groups with torsion]
	Let $Y$ be the wedge of a hyperbolic surface of genus two $S$ and a sphere $\mathbb{S}^2$ and $X$ be its universal cover, which is Gromov-hyperbolic. Denote by $G_n$ the group of isometries of $Y$ generated by the isometry that fixes $S$ and acts as a rotation of angle $\frac{2\pi}{n}$ on $\mathbb{S}^2$ fixing the wedging point. Let $\Gamma_n$ be the covering group of $G_n$ acting on $X_n := X$ by isometries. The action of $\Gamma_n$ is clearly discrete with bounded codiameter. However $\Gamma_\omega$ is not discrete.
\end{ex}

\begin{ex}[Gromov-hyperbolicity]
	Let $X=\mathbb{R}^2$, $x = (0,0)$ and $\Gamma_n$ be the cocompact, discrete, torsion-free group generated by the translations of vectors $(\frac{1}{n},0)$ and $(0,1)$. It is clear that $\Gamma_\omega$ is not discrete for every non-principal ultrafilter $\omega$.
\end{ex}

%

		\bibliographystyle{alpha}
		\bibliography{Continuity_of_critical_exponent_new}

\begin{thebibliography}{BCGS21}

\bibitem[BCGS17]{BCGS}
G.~Besson, G.~Courtois, S.~Gallot, and A.~Sambusetti.
\newblock Curvature-free margulis lemma for gromov-hyperbolic spaces.
\newblock {\em arXiv preprint arXiv:1712.08386}, 2017.

\bibitem[BCGS21]{BCGS2}
Gérard Besson, Gilles Courtois, Sylvestre Gallot, and Andrea Sambusetti.
\newblock Finiteness theorems for gromov-hyperbolic spaces and groups.
\newblock {\em arXiv preprint arXiv:2109.13025}, 2021.

\bibitem[BH13]{BH09}
M.~Bridson and A.~Haefliger.
\newblock {\em Metric spaces of non-positive curvature}, volume 319.
\newblock Springer Science \& Business Media, 2013.

\bibitem[BJ97]{BJ97}
C.~Bishop and P.~Jones.
\newblock Hausdorff dimension and kleinian groups.
\newblock {\em Acta Mathematica}, 179(1):1--39, 1997.

\bibitem[BL12]{BL12}
A.~Bartels and W.~Lück.
\newblock Geodesic flow for $\mathrm{CAT}(0)$–groups.
\newblock {\em Geom. Topol.}, 16(3):1345--1391, 2012.

\bibitem[Cav21a]{Cav21}
N.~Cavallucci.
\newblock Entropies of non positively curved metric spaces.
\newblock {\em arXiv preprint arXiv:2102.07502}, 2021.

\bibitem[Cav21b]{Cav21bis}
Nicola Cavallucci.
\newblock Topological entropy of the geodesic flow of non-positively curved
  metric spaces.
\newblock {\em arXiv preprint arXiv:2105.11774}, 2021.

\bibitem[CDP90]{CDP90}
M.~Coornaert, T.~Delzant, and A.~Papadopoulos.
\newblock {\em G{\'e}om{\'e}trie et th{\'e}orie des groupes: les groupes
  hyperboliques de Gromov}.
\newblock Lecture notes in mathematics. Springer-Verlag, 1990.

\bibitem[Coo93]{Coo93}
M.~Coornaert.
\newblock Mesures de patterson-sullivan sur le bord d'un espace hyperbolique au
  sens de gromov.
\newblock {\em Pacific J. Math.}, 159(2):241--270, 1993.

\bibitem[CS20]{CavS20bis}
N.~Cavallucci and A.~Sambusetti.
\newblock Discrete groups of packed, non-positively curved, gromov hyperbolic
  metric spaces.
\newblock {\em arXiv preprint arXiv:2102.09829}, 2020.

\bibitem[CS21]{CavS20}
Nicola Cavallucci and Andrea Sambusetti.
\newblock Packing and doubling in metric spaces with curvature bounded above.
\newblock {\em Mathematische Zeitschrift}, pages 1--46, 2021.

\bibitem[DK18]{DK18}
Cornelia Dru{\c{t}}u and Michael Kapovich.
\newblock {\em Geometric group theory}, volume~63.
\newblock American Mathematical Soc., 2018.

\bibitem[DSU17]{DSU17}
T.~Das, D.~Simmons, and M.~Urba{\'n}ski.
\newblock {\em Geometry and dynamics in Gromov hyperbolic metric spaces},
  volume 218.
\newblock American Mathematical Soc., 2017.

\bibitem[Fuk86]{Fuk86}
K.~Fukaya.
\newblock Theory of convergence for riemannian orbifolds.
\newblock {\em Japanese journal of mathematics. New series}, 12(1):121--160,
  1986.

\bibitem[Gro81]{Gr81}
M.~Gromov.
\newblock Groups of polynomial growth and expanding maps (with an appendix by
  jacques tits).
\newblock {\em Publications Math\'ematiques de l'IH\'ES}, 53:53--78, 1981.

\bibitem[Her16]{Her16}
David~A Herron.
\newblock Gromov--hausdorff distance for pointed metric spaces.
\newblock {\em The Journal of Analysis}, 24(1):1--38, 2016.

\bibitem[Jan17]{Jan17}
D.~Jansen.
\newblock Notes on pointed gromov-hausdorff convergence.
\newblock {\em arXiv preprint arXiv:1703.09595}, 2017.

\bibitem[Kel17]{Kel17}
John~L Kelley.
\newblock {\em General topology}.
\newblock Courier Dover Publications, 2017.

\bibitem[Pau96]{Pau96}
F.~Paulin.
\newblock Un groupe hyperbolique est déterminé par son bord.
\newblock {\em Journal of the London Mathematical Society}, 54(1):50--74, 1996.

\bibitem[Pau97]{Pau97}
F.~Paulin.
\newblock On the critical exponent of a discrete group of hyperbolic
  isometries.
\newblock {\em Differential Geometry and its Applications}, 7(3):231--236,
  1997.

\end{thebibliography}
		
	\end{document}